\documentclass[a4paper,12pt,reqno]{amsart}
\usepackage{a4wide}
\usepackage{amsmath}
\usepackage{amssymb}
\usepackage{amsthm}
\usepackage{latexsym}
\usepackage{graphicx}
\usepackage[english]{babel}

\newtheorem{teor}[subsection]{Theorem}
\newtheorem{lema}[subsection]{Lemma}

\theoremstyle{definition}

\theoremstyle{remark}

\def\qdepth{\operatorname{hdepth}}
\def\hdepth{\operatorname{hdepth}}

\numberwithin{equation}{section}

\begin{document}

\title[Comparing Hilbert depth of $I$ with Hilbert depth of $S/I$. II]{Comparing Hilbert depth of $I$ with Hilbert depth of $S/I$. II}
\author[Andreea I.\ Bordianu, Mircea Cimpoea\c s 
       ]
  {Andreea I.\ Bordianu$^1$ and Mircea Cimpoea\c s$^2$
	}
\date{}

\keywords{Depth, Hilbert depth, monomial ideal, squarefree monomial ideal}

\subjclass[2020]{05A18, 06A07, 13C15, 13P10, 13F20}

\footnotetext[1]{ \emph{Andreea I.\ Bordianu}, University Politehnica of Bucharest, Faculty of
Applied Sciences, 
Bucharest, 060042, E-mail: andreea.bordianu@stud.fsa.upb.ro}
\footnotetext[2]{ \emph{Mircea Cimpoea\c s}, University Politehnica of Bucharest, Faculty of
Applied Sciences, 
Bucharest, 060042, Romania and Simion Stoilow Institute of Mathematics, Research unit 5, P.O.Box 1-764,
Bucharest 014700, Romania, E-mail: mircea.cimpoeas@upb.ro,\;mircea.cimpoeas@imar.ro}

\begin{abstract}
Let $I$ be a squarefree monomial ideal of $S=K[x_1,\ldots,x_n]$. We prove that if $\hdepth(S/I)\leq 6$ of $n\leq 9$
then $\hdepth(I)\geq \hdepth(S/I)$, giving a positive answer to a problem putted in \cite{bordi}.
\end{abstract}

\maketitle

\section{Introduction}

Let $K$ be a field and $S=K[x_1,\ldots,x_n]$ the polynomial ring over $K$.
Let $M$ be a finitely generated graded $S$-module. The Hilbert depth of $M$, denoted by $\hdepth(M)$, is the 
maximal depth of a finitely generated graded $S$-module $N$ with the same Hilbert series as $M$.
For the basic properties of this invariant we refer the reader to \cite{bruns,uli}.

Let $0\subset I\subsetneq J\subset S$ be two squarefree monomial ideals.
In \cite{lucrare2} we presented a new method of computing the Hilbert depth of a quotient of $J/I$, as follows:

For all $0\leq j\leq n$, we let $\alpha_j(J/I)$ be the number of squarefree monomials 
$u$ of degree $d$ with $u\in J\setminus I$. For all $0\leq q\leq n$ and $0\leq k\leq q$, we consider the integers
\begin{equation*}\label{betak}
  \beta_k^q(J/I):=\sum_{j=0}^k (-1)^{k-j} \binom{q-j}{k-j} \alpha_j(J/I).
\end{equation*}
In \cite[Theorem 2.4]{lucrare2} we proved that the Hilbert depth of $J/I$ is
$$\qdepth(J/I):=\max\{d\;:\;\beta_k^d(J/I) \geq 0\text{ for all }0\leq k\leq d\}.$$
Using this combinatorial characterization of the Hilbert depth and the Kruskal-Katona theorem, in \cite{bordi} 
we studied connections between $\hdepth(S/I)$ and $\hdepth(I)$. We proved that if $I\subset S$ is a 
squarefree monomial ideal with $\hdepth(S/I)\leq 3$ then
$\hdepth(I)\geq 4$; see \cite[Theorem 3.10]{bordi}. Also, we proved that if $\hdepth(S/I)\leq 5$ then $\hdepth(I)\geq \hdepth(I)$;
see \cite[Theorem 3.15]{bordi} and \cite[Theorem 4.4]{bordi}. We conjectured that if $\hdepth(S/I)\leq 6$ or $n\leq 9$ then 
$\hdepth(I)\geq \hdepth(S/I)$. We give a positive answer to this question; see Theorem \ref{main2} and Theorem \ref{main3}.

\section{Main results}

Let $I\subset S=K[x_1,\ldots,x_n]$ be a squarefree monomial ideal with $\qdepth(S/I)=q\leq n-1$.
For convenience, we denote 
$$\alpha_j=\alpha_j(S/I), 0\leq j\leq n,\text{ and }\beta_k^d=\beta_k^d(S/I),\; 0\leq k\leq d\leq n.$$
First, we recall the following result:

\begin{teor}(see \cite[Theorem 2.2]{bordi})\label{teo1}
The following are equivalent:
\begin{enumerate}
\item[(1)] $I$ is principal.
\item[(2)] $\qdepth(I)=n$.
\item[(3)] $\qdepth(S/I)=n-1$.
\end{enumerate}
\end{teor}

Theorem \ref{teo1} shows that, in order to show that $\qdepth(I)\geq \qdepth(S/I)$, it is safe to assume that
$I$ is not principal. Hence $n\geq q+2$. Another important reduction which can be done is to assume that $I\subset \mathfrak m^2$;
see \cite[Remark 2.4]{bordi}. Hence $\alpha_0=1$ and $\alpha_1=n$.

Let $N$ and $k$ be two positive integers. Then $N$ can be uniquely written as 
$$N=\binom{n_k}{k}+\binom{n_{k-1}}{n_{k-1}}+\cdots+\binom{n_j}{j},\text{ where }n_k>n_{k-1}>\cdots>n_j\geq j\geq 1.$$
As a direct consequence of the Kruskal-Katona Theorem (\cite[Theorem 2.1]{stanley}), we have the following restrictions on $\alpha_j$'s:

\begin{lema}\label{cord}
If $\alpha_k=\binom{n_k}{k} + \binom{n_{k-1}}{k-1}+\cdots+\binom{n_j}{j}$, as above, where $2\leq k\leq n-1$, then
\begin{enumerate}
\item[(1)] $\alpha_{k-1}\geq \binom{n_k}{k-1}+\binom{n_{k-1}}{k-2}+\cdots+\binom{n_j}{j-1}$.
\item[(2)] $\alpha_{k+1}\leq \binom{n_k}{k+1}+\binom{n_{k-1}}{k}+\cdots+\binom{n_j}{j+1}$.
\end{enumerate}
\end{lema}

Also, we have the following result:

\begin{lema}(See \cite[Lemma 3.13]{bordi} and \cite[Proposition 3.14]{bordi})\label{lem}
Let $I\subset S$ be a proper squarefree monomial ideal with $\qdepth(S/I)=q$.
The following are equivalent:
\begin{enumerate}
\item[(1)] $\qdepth(I)\geq \qdepth(S/I).$
\item[(2)] $\beta_{k}^{q} \leq \binom{n-q+k-1}{k},\text{ for all }3\leq k\leq q$.
\end{enumerate}
\end{lema}

We recall also the following lemma:

\begin{lema}(See \cite[Lemma 4.1]{bordi})\label{b3q}
Let $I\subset S$ be a proper squarefree monomial ideal.
\begin{enumerate}
\item[(1)] If $q=6$ then $\beta_3^6(S/I)\leq \binom{n-4}{3}$.
\item[(2)] If $q=7$ then $\beta_3^7(S/I)\leq \binom{n-5}{3}$.
\end{enumerate}
\end{lema}

In the following, our aim is to tackle the open problems from \cite{bordi}.

\begin{lema}\label{b46}
With the above notations, we have that $\beta_4^6\leq \binom{n-3}{4}$.
\end{lema}

\begin{proof}
First note that $n\geq 8$. Since $\beta_2^6\geq 0$ and $\beta_3^6\geq 0$ it follows that
\begin{equation}\label{conditie}
\alpha_2\geq 5(n-3),\;\alpha_3\geq 4\alpha_2-10(n-2)\geq 10(n-4),\;4\alpha_2\leq \alpha_3+10(n-2).
\end{equation}
We consider the functions
$$
f(x)=\binom{x}{4}-3\cdot \binom{x}{3},\; g(x)=\binom{x}{3}-3\cdot \binom{x}{2}\text{ and }
h(x)=\binom{x}{2}-3\cdot \binom{x}{1}.$$
Note that $f(x)=\frac{1}{24}x(x-1)(x-2)(x-15)$ and we have:

\begin{table}[htb]
\centering
\begin{tabular}{|l|l|l|l|l|l|l|l|l|l|l|l|l|l|l|l|l|l|l|l|}
\hline
$x$    & 1 & 2 &  3 & 4   & 5   & 6   & 7   & 8   & 9    & 10   & 11    & 12   & 13   & 14  & 15 & 16  \\ \hline
$f(x)$ & 0 & 0 & -3 & -11 & -25 & -45 & -70 & -98 & -126 & -150 & -165  & -165 & -143 & -91 & 0  & 140 \\ \hline
\end{tabular}
\end{table}
Also, $f$ is decreasing on $[2,11]$ and is increasing on $[13,\infty)$.

We have that $g(x)=\frac{1}{6}x(x-1)(x-11)$ and:
\begin{table}[htb]
\centering
\begin{tabular}{|l|l|l|l|l|l|l|l|l|l|l|l|l|l|l|l|l|l|l|l|l|}
\hline
$x$    & 1 &  2 &  3 &  4  & 5   & 6   & 7   & 8   & 9   &  10 & 11 & 12 & 13 \\ \hline
$g(x)$ & 0 & -3 & -8 & -14 & -20 & -25 & -28 & -28 & -24 & -15 & 0  & 22 & 52 \\ \hline
\end{tabular}
\end{table}

Also, $g$ is decreasing on $[1,7]$ and is increasing on $[8,\infty)$.

We have that $h(x)=\frac{1}{2}x(x-7)$ and:
\begin{table}[htb]
\centering
\begin{tabular}{|l|l|l|l|l|l|l|l|l|l|l|l|l|l|l|l|l|l|}
\hline
$x$    &  1 & 2  & 3  & 4  & 5  &  6 & 7 & 8 & 9 & 10 & 11 & 12\\ \hline
$h(x)$ & -3 & -5 & -6 & -6 & -5 & -3 & 0 & 4 & 9 & 15 & 22 & 30\\ \hline
\end{tabular}
\end{table}

Also, $h$ is increasing on $[3,\infty)$.

We consider several cases:
\begin{enumerate}
\item[(a)] $\alpha_3=\binom{n_3}{3}$. From Lemma \ref{cord} it follows that $\alpha_4\leq \binom{n_3}{4}$ and $\alpha_2\geq \binom{n_3}{2}$.
If $n_3=n$ then $\alpha_2=\binom{n}{2}$ and therefore
$$\beta_4^6 \leq \binom{n}{4}-3\binom{n}{3}+6\binom{n}{2}-10n+15 = \binom{n-3}{4},$$
as required. Hence, we can assume that $n_3<n$. 

Since $\alpha_2\leq \binom{n}{2}$ and $\alpha_4\leq \binom{n_3}{4}$,
a sufficient condition to have $\beta_4^6 \leq \binom{n-3}{4}$ is $f(n_3)\leq f(n)$.
This is clearly satisfied for $n\geq 15$; see the table with values of $f(x)$. 
Hence, we may assume $n\leq 14$. We consider several subcases:
\begin{itemize}
\item $n=14$. From \eqref{conditie} it follows that
$\alpha_3\geq 10(14-4)=100$ and thus $n_3\geq 10$. In particular, we get $f(n_3)\leq f(10)=-150<f(14)=-91$ and we are done.
\item $n=13$. From \eqref{conditie} it follows that $\alpha_3\geq 90$ and thus 
              $n_3\geq 10$ and again we have $f(n_3)\leq f(10)=-150<f(13)=-143$.
\item $n=12$. From \eqref{conditie} it follows that $\alpha_3\geq 80$ and thus $n_3\geq 9$. If $n_3=11$ then
              $f(n_3)=f(11)=f(12)$ and we are done, so we may assume that $n_3\in\{9,10\}$, i.e. $\alpha_3\in\{84,120\}$.
							If $n_3=9$, i.e. $\alpha_3=84$, then from \eqref{conditie} it follows that
							$$\alpha_2\geq 5\cdot 9=45\text{ and }4\alpha_2\leq 184.$$
							Therefore $\alpha_2\in\{45,46\}$. Since $\alpha_4\leq \binom{9}{4}=126$, it follows that
							$$\beta_4^6 \leq 126 - 3\cdot 84 + 6 \cdot 46 - 10\cdot 12 + 15 = 45 \leq 126 = \binom{12-3}{4}.$$
							Similarly, if $n_3=10$, i.e. $\alpha_3=120$, then $4\alpha_2 \leq 220$, that is $\alpha_2\leq 55$.
							Since $\alpha_4\leq \binom{10}{4}=210$, it follows that
							$$\beta_4^6 \leq 210 - 3\cdot 120 + 6 \cdot 55 - 10\cdot 12 + 15 = 75 \leq 126 = \binom{12-3}{4}.$$
							Hence, we are done.
\item $n=11$. From \eqref{conditie} it follows that $\alpha_3\geq 70$	and thus $n_3\geq 9$. If $n_3=9$ ($\alpha_3=84$),
              then from \eqref{conditie} it follows that
							$$\alpha_2\geq 40\text{ and }4\alpha_2\leq 174,\text{ i.e. }\alpha_2\leq 43.$$
							Since $\alpha_4\leq \binom{9}{4}$ it follows that
							$$\beta_4^6 \leq 126-3\cdot 84 + 6\cdot 43 - 110 + 15 = 37 \leq \binom{8}{4}=70,$$
							and we are done. If $n_3=10$ ($\alpha_3=120$) then, from \eqref{conditie}, it follows
							that $4\alpha_2\leq 210$, that is $\alpha_2\leq 52$. Hence
							$$\beta_4^6 \leq 210-3\cdot 120 + 6\cdot 52 - 110 + 15 = 67 \leq \binom{8}{4},$$
							and we are done again.
\item $n=10$. From \eqref{conditie} it follows that $\alpha_3\geq 60$	and thus $n_3=9$ and $\alpha_3=84$. Also $\alpha_4\leq 126$.
              Again, from \eqref{conditie} it follows that $4\alpha_2\leq 164$, i.e. $\alpha_2\leq 41$. Therefore
							$$\beta_4^6 \leq 	126 - 3\cdot 84 + 6\cdot 41 - 100 + 15 = 35 = \binom{10-3}{3},$$
							and we are done.
\item $n=9$. From \eqref{conditie} it follows that $\alpha_3\geq 50$ and thus $n_3=8$ and $\alpha_3=56$. Also, $\alpha_4\leq 70$.
             Again, from \eqref{conditie} it follows that $4\alpha_2\leq 126$ and $\alpha_2\leq 31$. It follows that
						 $$\beta_4^6 \leq  70 - 3\cdot 56 + 6 \cdot 31 - 90 + 15 = 13 \leq \binom{9-3}{4} = 15,$$
						 and we are done.
\item $n=8$. From \eqref{conditie} it follows that $\alpha_3\geq 40 > \binom{7}{3}=35$, a contradiction, since $n_3\leq 7$.					
\end{itemize}							
\item[(b)] $\alpha_3=\binom{n_3}{3}+\binom{n_2}{2}$, where $n>n_3>n_2\geq 2$.
           From Lemma \ref{cord} it follows that 
					 $$\alpha_4\leq \binom{n_3}{4}+\binom{n_2}{3}\text{ and }\alpha_2\geq \binom{n_3}{2}+\binom{n_2}{1}.$$
					 Since $\alpha_2\leq \binom{n}{2}$, in order to have $\beta_4^6\leq \binom{n-3}{4}$ it suffices
					 $f(n_3)+g(n_2)\leq f(n)$. If $n\geq 16$ then
					 $$f(n)-(f(n_3)+g(n_2))\geq f(n)-(f(n-1)+g(n-2))=\frac{1}{2}(n-2)(n-9)>0,$$
					 and we are done. If $n = 15$ then, looking at the tables of values of $f(x)$ and $g(x)$ it is easy to
					 check that $f(n_3)+g(n_2)\leq 0 = f(15)$, so we are done also in this case. Hence, we may assume $n\leq 14$. 
					 We consider several subcases:
					 \begin{itemize}
					 \item $n=14$. As in the case (a), we have $\alpha_3\geq 100$. Thus $n_3\geq \{9,10,11,12,13\}$. It follows that
					       $$f(n_3)+g(n_2)\leq f(13)+g(12)=-126+22=-104<-91=f(14),$$
								 and we are done.
					 \item $n=13$. As in the case (a), we have $\alpha_3\geq 90$. Thus $n_3\in \{9,10,11,12\}$. If $n_3\geq 10$ then					       
					       $$f(n_3)+g(n_2)\leq f(10)+g(11)=-150<-143=f(13),$$
								 and we are done. Now, assume that $n_3=9$. If $n_2\geq 5$ then 
								 $$f(n_3)+g(n_2)\leq f(9)+g(5)=-146<-145=f(13),$$
								 and we are done again. If $n_2=4$ then $\alpha_3=90$. From \eqref{conditie} it follows that
								 $4\alpha_2 \leq \alpha_3 + 110 = 200$, that is $\alpha_2\leq 50$. If follows that
								 $$\beta_4^6 \leq f(9)+g(4) + 6\cdot 50 - 10 \cdot 13 + 15 = 45 < \binom{10}{3}.$$
								 If $n_2\leq 3$ then $\alpha_3\leq 87$, a contradiction.								
					 \item $n=12$. As in the case (a), we have $\alpha_3\geq 80$ and $n_3\in \{9,10,11\}$. If $n_3=11$ then $n_2\leq 10$ and
					       therefore $f(n_3)+g(n_2)\leq f(n_3)=f(11)=f(12)$, as required. If $n_3=9$ then $\alpha_3\leq \binom{9}{3}+\binom{8}{2}=112$.
								 From \eqref{conditie} it follows that $4\alpha_2 \leq 212$, that is $\alpha_2\leq 53$. It follows that
								 \begin{align*}
								 & \beta_4^6 \leq f(n_3)+g(n_2) - 3\cdot 85 + 6\cdot 53 - 10\cdot 12 + 15 = \\
								 & = f(9)+g(n_2)-42 =-166+g(n_2)\leq -166< f(12)=-165, 
								 \end{align*}
								 and we are done.
								 If $n_3=10$ then $n_2\leq 9$. If $n_2\geq 5$ then
								 $$f(n_3)+g(n_2)\leq f(10)+g(5) = -170 < -165 = f(12),$$
								 and we are done. If $n_2\leq 4$ then $121\leq \alpha_3 \leq \binom{10}{3}+\binom{4}{2} = 126$. From \eqref{conditie}
								 it follows that $4\alpha_2 \leq 226$, that is $\alpha_2\leq 56$. Since $\alpha_4\leq \binom{10}{4}+\binom{4}{3}=214$, 
								 it follows that
							   $$\beta_4^6 \leq 214 - 3\cdot 121 + 6 \cdot 56 - 10\cdot 12 + 15 = 82 \leq 126 = \binom{12-3}{4}.$$	
					 \item $n=11$. As in the case (a), we have $\alpha_3\geq 70$. Thus $n_3\in \{8,9,10\}$.								 
					       If $n_3=8$ then $n_2\geq 6$ and $72=\binom{8}{3}+\binom{6}{2}\leq \alpha_3\leq \binom{8}{3}+\binom{7}{2} = 77$.								 
								 In follows that $4\alpha_2\leq 77+90=167$ and thus $\alpha_2\leq 41$. Since $\alpha_4\leq \binom{8}{4}+\binom{7}{3}=112$,
								 we get
								 $$\beta_4^6 \leq 112 - 4\cdot 72 + 6\cdot 41 - 10\cdot 11 + 15 = -25,$$
					       a contradiction.
                 If $n_3=9$ then $85\leq \alpha_3\leq \binom{9}{3}+\binom{8}{2}=112$ and thus
								 $$4\alpha_2\leq 112+90=202,\text{ that is }\alpha_2\leq 50.$$
								 Since $\alpha_4\leq \binom{9}{4}+\binom{8}{3}=182$, it follows that
								 $$\beta_4^6 \leq 182 - 4\cdot 85 + 6\cdot 50 - 10\cdot 11 + 15 = 47 \leq \binom{8}{4}=70.$$
								 If $n_3=10$ and $n_2\geq 5$ then $f(n_3)+g(n_2)\leq -170 < f(11)=-165$ and there is nothing to prove.
								 Hence we may assume $n_2\leq 4$. If $n_2=4$ then $\alpha_3=126$ and, from \eqref{conditie}, it follows
								 that $\alpha_2\leq 54$. Since $\alpha_4-3\alpha_3\leq f(10)+g(4)=-164$, it follows that
								 $$\beta_4^6 \leq - 164+6\cdot 54 - 10\cdot 11 + 15 = 65 < 70=\binom{8}{4}.$$
								 If $n_2\in \{2,3\}$ then $\alpha_3\leq 123$ and therefore $\alpha_2\leq 53$. Since 
								 $\alpha_4-3\alpha_3\leq f(10)+g(2)=-153$, it follows that
								 $$\beta_4^6 \leq - 153+6\cdot 53 - 10\cdot 11 + 15 = 70=\binom{8}{4}.$$														
					 \item $n=10$. As in the case (a), we have $\alpha_3\geq 60$. Thus $n_3\in \{8,9\}$. Assume $n_3=9$.					 
					       If $n_2\in\{6,7\}$ then
					       $f(n_3)+g(n_2)\leq -151 < -150=f(10)$ and we are done. So we may assume $n_2\leq 5$.
								 If $n_2=5$ then $\alpha_3= 94$ and therefore $4\alpha_2 \leq 94+80=174$, i.e $\alpha_2\leq 43$.
								 Since $\alpha_4-3\alpha_3\leq f(n_3)+g(n_2)=f(9)+g(5)=-146$, 
								 it follows that 
								 $$\beta_4^6 \leq -146 +6\cdot 43 - 100 + 15 = 27<35 = \binom{7}{4}.$$
								 If $n_2=4$ then $\alpha_3= 90$ and therefore $\alpha_2\leq 42$. Similarly, we get $\beta_4^6\leq 27$.								 
								 If $n_2\leq 3$ then $\alpha_3\leq 87$ and therefore $\alpha_2\leq 41$. We get 
								 $$\beta_4^6\leq f(9)+g(2)+6\cdot 41 - 100 + 15 =32<\binom{7}{4}.$$
								 Now, assume that $n_3=8$. If $n_2\leq 3$ then $\alpha_3\leq \binom{8}{4}+\binom{3}{2}=59$, a contradiction.
								 Thus $n_2\in\{4,5,6,7\}$. If $n_2=4$ then $\alpha_3=66$ and thus $4\alpha_2\leq 142$, that is $\alpha_2\leq 35$.
								 It follows that $$\beta_4^6 \leq f(9)+g(4) + 6\cdot 35 - 100 + 15 = 3 < \binom{7}{4}.$$
								 If $n_2\geq 5$ then $\alpha_3\leq 77$ and thus $\alpha_2\leq 39$.
								 It follows that $$\beta_4^6 \leq f(9)+g(5) + 6\cdot 39 - 100 + 15 = 31 < \binom{7}{4}.$$
					 \item $n=9$. As in the case (a), we have $\alpha_3\geq 50$ and $n_3 \in \{7,8\}$. Assume that $n_3=8$. If $n_2=7$ then 
					       $$f(n_3)+g(n_2)=f(8)+g(7)=-98-28=-126=f(9),$$
								 and we are done. If $n_2=6$ then $\alpha_3 = \binom{8}{3}+\binom{6}{2}=71$ and $4\alpha_2\leq 141$,
								 that is $\alpha_2\leq 35$. Thus
								 $$\beta_4^6 \leq f(8)+g(6)+6\cdot 35 - 90 + 15 = 12 < \binom{6}{4} = 15,$$
								 and we are done. If $n_2=5$ then $\alpha_3 = \binom{8}{3}+\binom{5}{2}=66$ and $\alpha_2\leq 34$. Thus
								 $$\beta_4^6 \leq f(8)+g(5)+6\cdot 34 - 90 + 15 = 11 < \binom{6}{4} = 15.$$
								 If $n_2=4$ then $\alpha_3=62$ and $\alpha_2 \leq 33$. Thus
								 $$\beta_4^6 \leq f(8)+g(4)+6\cdot 32 - 90 + 15 = 11 < \binom{6}{4} = 15.$$
								 If $n_2=3$ then $\alpha_3=59$ and $\alpha_2\leq 32$. Similarly, we get $\beta_4^6\leq 11$.
								 If $n_2=2$ then $\alpha_3=57$ and $\alpha_2\leq 31$. Similarly, we get $\beta_4^6\leq 10$.								
								 Now, assume that $n_3=7$. It follows that $n_2=6$ and $\alpha_3=50$. Thus $\alpha_2\leq 30$ and
								 $$\beta_4^6 \leq f(7)+g(6)+6\cdot 30 - 90 + 15 = 10<15.$$								
					 \item $n=8$. As in the case (a), we have $\alpha_3\geq 40$. Thus $n_3=7$ and $n_2\in\{4,5,6\}$.
					       If $n_2=4$ then $\alpha_3=41$ and $\alpha_2\leq 25$. It follows that
								 $$\beta_4^6 \leq f(7)+g(4)+6\cdot 25 - 80+15 = 1 \leq \binom{5}{4}=5.$$
								 If $n_2=5$ then $\alpha_3=45$ and $\alpha_2\leq 26$. It follows that
								 $$\beta_4^6 \leq f(7)+g(5)+6\cdot 26 - 80+15 = 1 \leq \binom{5}{4}.$$
								 If $n_2=6$ then $\alpha_3=50$ and $\alpha_2\leq 27$. It follows that
								 $$\beta_4^6 \leq f(7)+g(6)+6\cdot 27 - 80+15 = 2 \leq \binom{5}{4}.$$
					 \end{itemize}

\item[(c)] $\alpha_3=\binom{n_3}{3}+\binom{n_2}{2}+\binom{n_1}{1}$, where $n>n_3>n_2>n_1\geq 1$.
           From Lemma \ref{cord} it follows that 
					 $$\alpha_4\leq \binom{n_3}{4}+\binom{n_2}{3}+\binom{n_1}{2}\text{ and }\alpha_2\geq \binom{n_3}{2}+\binom{n_2}{1}+1.$$
					 Since $\alpha_2\leq \binom{n}{2}$, in order to have $\beta_4^6\leq \binom{n-3}{4}$ it suffices
					 $f(n_3)+g(n_2)+h(n_1)\leq f(n)$. If $n\geq 16$ then 
					 $$f(n)-(f(n_3)+g(n_2)+h(n_1))\geq f(n)-(f(n-1)+g(n-2)+h(n-3))=n-6>0.$$
					 If $n = 15$ then, looking at the tables of values of $f(x)$ and $g(x)$ it is easy to
					 check that $f(n_3)+g(n_2)+h(n_1)\leq 0 = f(15)$, so we are done also in this case. Hence, we may assume $n\leq 14$. 
					 We consider several subcases:
					 \begin{itemize}
					 \item $n=14$. As in the case (b), we have $\alpha_3\geq 100$ and $n_3\geq \{9,10,11,12,13\}$. If $n_3\geq 10$ then
					       $$f(n_3)+g(n_2)+h(n_1)\leq f(13)+g(12)+h(11)=-143+22+22=-99<-91=f(14),$$
								 and we are done.	If $n_3=9$ then $n_2\leq 8$ and $n_1\leq 7$ and thus 
								 $$f(n_3)+g(n_2)+h(n_1)\leq f(9) = -126 < -91=f(14).$$
					 \item $n=13$. As in the case (b), we have $\alpha_3\geq 90$ and $n_3\in \{9,10,11,12\}$.
					       Assume that $n_3=9$. If $n_2\geq 4$ then $$f(n_3)+g(n_2)+h(n_1) \leq f(9)+g(4)+h(1) = -143 = f(13),$$
								 and we are done. If $n_2\leq 3$ then $\alpha_3\leq \binom{9}{3}+\binom{3}{2}+\binom{2}{1}=89$, a contradiction.												
					       If $n_3=10$ then $n_2\leq 9$ and $n_1\leq 8$. In particular, $g(n_2)\leq 0$ and $h(n_3)\leq 4$. Therefore								
					       $$f(n_3)+g(n_2)+h(n_1)\leq f(10)+4=-150+4=-146<-143=f(13),$$
								 and we are done.	Similarly, if $n_3\in\{11,12\}$ then $n_2\leq 11$ and $n_2\leq 10$. Therefore
								 $$f(n_3)+g(n_2)+h(n_1)\leq -165+0+15 = -150 < - 143 =f(13),$$
								 and we are done again.								
					 \item $n=12$. As in the case (b), we have $\alpha_3\geq 80$ and $n_3\in \{8,9,10,11\}$. If $n_3=11$ then 
					       $n_2\leq 10$ and $n_1\leq n_2-1$.
					       If is easy to check that $g(n_2)+h(n_1)\leq 0$ and therefore $$f(n_3)+g(n_2)+h(n_1)\leq f(n_3)=f(11)=f(12),$$ as required. 	
								 Assume $n_3=9$.
								 It follows that $\alpha_3\leq \binom{9}{3}+\binom{8}{2}+\binom{7}{1}=119$. From \eqref{conditie} it follows that
								 $4\alpha_2\leq 219$, that is $\alpha_2\leq 54$. Also, $n_2\leq 8$, $n_1\leq 7$ and $\alpha_3\geq \binom{9}{3}+2=86$.
								 It follows that
								 \begin{align*}
								 & \beta_4^6 \leq f(n_3)+g(n_2)+h(n_1) - 3\cdot 86 + 6\cdot 54 - 10\cdot 12 + 15 = \\
								 & = f(9)+g(n_2)+h(n_1)-39 =-165+g(n_2)+h(n_1) < -165=f(12),
								 \end{align*}
								 and we are done. If $n_3=10$ then $n_2\leq 9$ and $n_1\leq 8$. As in the case (b), if $n_2\geq 5$ then
								 $$f(n_3)+g(n_2)+h(n_1)\leq f(10)+g(5)+h(8) = -166 < -165 = f(12),$$
								 and we are done. If $n_2\leq 4$ then $n_1\leq 3$ and $122\leq \alpha_3 \leq \binom{10}{3}+\binom{4}{2} + \binom{3}{1} = 129$. From \eqref{conditie}
								 it follows that $4\alpha_2 \leq 229$, that is $\alpha_2\leq 57$. Since 
								 $\alpha_4\leq \binom{10}{4}+\binom{4}{3}+\binom{3}{2}=217$, it follows that
							   $$\beta_4^6 \leq 217 - 3\cdot 122 + 6 \cdot 57 - 10\cdot 12 + 15 = 88 \leq 126 = \binom{12-3}{4}.$$
								 If $n_3=8$ then the condition $\alpha_3\geq 80$ implies $n_2=7$ and $3\leq n_1\leq 6$.
								 In particular, we have $\alpha_3\leq 83$ and thus $\alpha_2\leq 45$. Also, we have 
								 $$\alpha_4-3\alpha_3\leq f(n_3)+g(n_2)+h(n_1) \leq f(8)+g(7)+h(6)=-129.$$
								 If follows that
								 $$\beta_4^6 \leq -129 + 6\cdot 45 - 120 + 15 = 36 \leq \binom{9}{3}=84.$$								 
					 \item $n=11$. As in the case (b), we have $\alpha_3\geq 70$ and $n_3\in \{8,9,10\}$.
                 First, assume that $n_3=8$. As in the case $n=12$ above, it follows that $\alpha_3\leq 83$. We get $\alpha_2\leq 43$.
								 If $n_2=7$ then, as above, we have $\alpha_4-3\alpha_3\leq -129$. Therefore
								 $$\beta_4^6 \leq -129 + 6\cdot 43 - 110 + 15 = 34 \leq \binom{8}{3}=56.$$
								 If $n_2\leq 6$ then $\alpha_3\leq 76$ and $\alpha_2\leq 41$. Since 
								 $$\alpha_4-3\alpha_2\leq f(n_3)+g(n_2)+h(n_1) \leq f(8)+g(2)+h(1)=-104,$$
								 it follows that $$\beta_4^6 \leq -104 + 6\cdot 41 - 110 + 15 = 46 \leq \binom{8}{3}=56.$$								
                 Now, assume $n_3\geq 9$. If $n_3=9$ and $n_2=8$ then 
								 $$113=\binom{9}{3}+\binom{8}{2}+\binom{1}{1} \leq \alpha_3\leq \binom{9}{3}+\binom{8}{2}+\binom{7}{1}=119,$$ 
								 and thus
								 $$4\alpha_2\leq 117+90=209,\text{ that is }\alpha_2\leq 52.$$
								 Since $\alpha_4\leq \binom{9}{4}+\binom{8}{3}+\binom{7}{2}=203$, it follows that
								 $$\beta_4^6 \leq 203 - 4\cdot 113 + 6\cdot 52 - 10\cdot 11 + 15 < 0,$$
								 a contradiction. If $n_3=8$ and $n_2\leq 7$ then $86 \leq \alpha_3 \leq 111$. 
								 It follows that $4\alpha_2\leq 201$ and thus $\alpha_2\leq 50$. On the other hand,
								 we have $\alpha_4\leq \binom{9}{4}+\binom{7}{3}+\binom{6}{2}=176$. Therefore
								 $$\beta_4^6 \leq 176 - 4\cdot 86 + 6\cdot 50 - 10\cdot 11 + 15 = 37<\binom{8}{4}.$$
								 Now, assume $n_3=10$. As in the case (b), we can assume $n_2\leq 4$ and $n_1\leq 3$.
								 Moreover, if $n_2=4$ then 
								 $$f(n_3)+g(n_2)+h(n_1)\leq f(10)+g(4)+h(1) = -167 < -165=f(11),$$
								 and we are done. So we may assume $n_2\leq 3$. It follows that $\alpha_3\leq 125$.
								 Hence, $4\alpha_2\leq 215$ and thus $\alpha_2\leq 53$. Since
								 $$\alpha_4-3\alpha_3\leq f(n_3)+g(n_2)+h(n_1) \leq f(10)+g(2)+h(1) = -156,$$
								 it follows that
								 $$\beta_4^6\leq -156+6\cdot 53 - 10\cdot 11 + 15 = 67 < \binom{8}{4}.$$								
					 \item $n=10$. As in the case (b), we have $\alpha_3\geq 60$ and $n_3\in \{8,9\}$. 
					       First, assume that $n_3=9$. If $n_2\geq 6$ we are done as in the case (b).
					       If $n_2=5$ and $n_1\geq 2$ then $$f(n_3)+g(n_2)+h(n_1)\leq f(9)+g(5)+h(2) = -151 < -150 = f(10),$$
								 and we are done. If $n_2=5$ and $n_1=1$ then $\alpha_3=95$ and $\alpha_2\leq 43$.
								 Since $\alpha_4-3\alpha_3\leq f(n_3)+g(n_2)+h(n_1)=f(9)+g(5)+h(1)=-149$, 
								 it follows that 
								 $$\beta_4^6 \leq -149 +6\cdot 43 - 100 + 15 = 24<35 = \binom{7}{3}.$$
								 If $n_2= 4$ then again $\alpha_2\leq 43$. Since 
								 $\alpha_4-3\alpha_3\leq f(9)+g(4)+h(1)=-143$, we are done using a similar argument.
								 If $n_3\leq 3$ then $\alpha_3\leq 88$ and $\alpha_2\leq 42$. Since 
								 $\alpha_4-3\alpha_3\leq f(9)+g(2)+h(1)=-132$, we are also get the required conclusion.	
															
                 Now, assume that $n_3=8$. It follows that $\alpha_3\leq 83$ and $\alpha_2\leq 40$.
								 If $n_2\geq 5$	then $$\alpha_4-4\alpha_3\leq f(n_3)+g(n_2)+h(n_1)\leq f(8)+g(5)+h(1) = -121$$
								 and therefore $$\beta_4^6 \leq -121 + 6\cdot 40 - 100 + 15 = 34 < 35=\binom{7}{3}.$$
								 If $n_2\leq 4$ then $\alpha_3\leq 65$ and $\alpha_2\leq 36$. It follows that
								 $$\beta_4^6 \leq f(8)+g(2)+h(1) + 6\cdot 36 - 100 + 15 = 27 < \binom{7}{3}.$$
																
					 \item $n=9$. As in the case (b), we have $\alpha_3\geq 50$ and $n_3 \in \{7,8\}$. 
					       Assume $n_3=8$. If $n_2=7$ then we are done
					       with the same argument as in the case (b). Similarly, if $n_2=6$ then 
								 $$f(n_3)+g(n_2)+h(n_1)\leq f(8)+g(6)+h(1) = -126 = f(9).$$
								 If $n_2=5$ then $\alpha_3\leq 70$ and $4\alpha_2\leq 140$, that is $\alpha_2\leq 35$.
								 Hence $$\beta_4^6 \leq f(8)+g(5)+h(1) + 6\cdot 35 - 90 + 15 = 14 < 15= \binom{6}{4}.$$
								 If $n_2=4$ then $\alpha_3\leq 65$ and thus $\alpha_2\leq 33$. Therefore
								 $$\beta_4^6 \leq f(8)+g(4)+h(1) + 6\cdot 33 - 90 + 15 = 8 \leq \binom{6}{4}.$$
								 If $n_2=3$ then $\alpha_3\leq 61$ and thus $\alpha_2\leq 32$. Similarly, we get $\beta_4^6\leq 8$.
								 If $n_2=2$ then $n_1=1$, $\alpha_3=58$ and $\alpha_2\leq 32$. Similarly, we get $\beta_4^6\leq 13$.
								
								 Now, assume that $n_3=7$. It follows that $n_2=6$ and $\alpha_3=50+n_1\leq 55$. Thus $\alpha_2\leq 31$ and
								 $$\beta_4^6 \leq f(7)+g(6)+h(1)+ 6\cdot 31 - 90 + 15 = 13 < 15.$$												
					 \item $n=8$. As in the previous cases, we have $\alpha_3\geq 40$. Thus $n_3=7$ and $n_2\geq 3$.
								 If $n_2=3$ then $n_2=1$, $\alpha_3=40$ and $\alpha_2\leq 25$.
								 It follows that
								 $$\beta_4^6 \leq f(7)+g(3)+h(1)+ 6\cdot 25 - 80+15 = 2 \leq \binom{5}{4}=5.$$								 								 
								 If $n_2=4$ then $\alpha_3 \leq 44$ and $\alpha_2\leq 26$. It follows that
								 $$\beta_4^6 \leq f(7)+g(4)+h(1)+6\cdot 26 - 80+15 = 4 \leq \binom{5}{4}=5.$$								 
								 If $n_2=5$ then $\alpha_3\leq 49$ and $\alpha_2\leq 27$. It follows that
								 $$\beta_4^6 \leq f(7)+g(5)+h(1)+6\cdot 27 - 80+15 = 4 \leq \binom{5}{4}.$$
								 If $n_2=6$ then $f(7)+g(6)+h(1) = -98 = f(8)$ and we are done.
					 \end{itemize}					
\end{enumerate}
\end{proof}

Note that, if $\qdepth(S/I)=7$ then the condition $\beta_4^7 \leq \binom{n-4}{4}$ do not hold in general; 
see Example \cite[Example 3.18]{bordi}.

\begin{lema}\label{b56}
With the above notations, we have that $\beta_5^6\leq \binom{n-2}{5}$.
\end{lema}

\begin{proof}
First, note that
\begin{align*}
& \beta_5^6 = \alpha_5-(2\beta_4^6+3\beta_3^6+4\beta_2^6+5\beta_1^6+6\beta_0^6) \leq \\
& \leq \alpha_5 -(5\beta_1^6+6\beta_0^6)=\alpha_5-(5(n-6)+6) = \alpha_5-(5n-24).
\end{align*}
If $\alpha_5 \leq \binom{n-2}{5} + (5n-24)$ then there is nothing to prove. Hence we may assume that
\begin{equation}\label{necesar}
\alpha_5 \geq \binom{n-2}{5}+5n-23.
\end{equation}
We denote $f_k(x)=\binom{x}{k}-2\binom{x}{k-1}$ for all $2\leq k\leq 5$. We have the following table of values:

\begin{table}[tbh]
\label{Tab:Tcr2}
\begin{tabular}{|l|l|l|l|l|l|l|l|l|l|l|l|l|l|l|l|l|l|}
\hline
$x$      & 1  & 2  & 3  & 4  & 5  & 6  & 7   & 8   &  9  & 10  & 11  & 12  & 13  & 14  \\ \hline
$f_5(x)$ & 0  & 0  & 0  & -2 & -9 & -24& -49 & -84 & -126& -168& -198& -198& -143&  0  \\ \hline
$f_4(x)$ & 0  & 0  & -2 & -7 & -15& -25& -35 & -42 & -42 & -30 & 0   & 55  & 143 & 273 \\ \hline
$f_3(x)$ & 0  & -2 & -5 & -8 & -10& -10& -7  & 0   & 12  & 30  & 55  & 88  & 130 & 182 \\ \hline
$f_2(x)$ & -2 & -3 & -3 & -2 & 0  & 3  & 7   & 12  & 18  & 25  & 33  & 42  & 52  & 63 \\ \hline
\end{tabular}
\end{table}

We claim that
\begin{equation}\label{claimm}
3\alpha_3-4\alpha_2\leq 3\binom{n}{3}-4\binom{n}{2}=\frac{1}{2}n(n-1)(n-6).
\end{equation}
If $\alpha_2=\binom{n}{2}$ then \eqref{claimm} is obviously true. Suppose that $\alpha_2=\binom{n_2}{2}$ for some $n_2<n$.
Then $\alpha_3\leq\binom{n_2}{3}$ and thus, since $n\geq 8$, we get
$$ 3\alpha_3-4\alpha_2\leq \frac{1}{2}n_2(n_2-1)(n_2-6) < \frac{1}{2}n(n-1)(n-6).$$
Similarly, if $\alpha_2=\binom{n_2}{2}+\binom{n_1}{1}$ for some $1\leq n_1<n_2<n$, since $n\geq 8$, we get
\begin{align*}
& 3\alpha_3-4\alpha_2\leq \frac{1}{2}n_2(n_2-1)(n_2-6)+ \frac{1}{2}n_1(3n_1-11) \leq \\
& \leq \frac{1}{2}\left( (n-1)(n-2)(n-7) + (n-2)(3n-17) \right) = \frac{1}{2}(n-2)(n^2-5n-10) = \\
& = \frac{1}{2}n(n-1)(n-6) - (6n-20) <\frac{1}{2}n(n-1)(n-6),
\end{align*}
and thus we proved the claim. Since $\beta_5^6=\alpha_5-2\alpha_4+3\alpha_3-4\alpha_2+5n-6$, 
from \eqref{claimm} it follows that in order to complete the proof, it suffices to show that 
$\alpha_5-2\alpha_4\leq f_5(n)$. If $\alpha_4=\binom{n}{4}$ this is obviously true, so we can dismiss this case.

Since $\qdepth(S/I)=6$ it follows that:
$$\beta_2^6=\alpha_2-5n+15\geq 0,\;\beta_3^6=\alpha_3-4\alpha_2+10n-20\geq 0.$$
As in the proof of Lemma \ref{b46}, we deduce that
\begin{equation}\label{cond1}
\alpha_2\geq 5(n-3),\;\alpha_3\geq 4\alpha_2-10(n-2)\geq 10(n-4),\;4\alpha_2\leq \alpha_3+10(n-2).
\end{equation}
Moreover, since $\beta_4^6=\alpha_4-3\alpha_3+6\alpha_2-10n+15\geq 0$ and using \eqref{cond1}, we deduce that
\begin{equation}\label{cond2}
\alpha_4\geq 3\alpha_3-6\alpha_2+10n-15, \; 3\alpha_3\leq \alpha_4+20n-75\text{ and }\alpha_4\geq 10n-45.
\end{equation}
We consider several cases:
\begin{enumerate}
\item[(a)] $\alpha_4=\binom{n_4}{4}$ with $n_4<n$. Note that $\alpha_5-2\alpha_4\leq f_5(n_4)$.
           From \eqref{necesar} it follows that $\alpha_4 >\binom{n-2}{4}$ and thus $n_4=n-1$. 					
           If $n\geq 12$ then $f_5(n_4)=f_5(n-1)\leq f_5(n)$ and we are done. Hence, we may assume that $n\leq 11$.									
           We consider several subcases:
					 \begin{itemize}
					 \item $n=11$. Since $\alpha_4=\binom{10}{4}=210$, from \eqref{cond2} it follows that $3\alpha_3 \leq 210 + 220- 75 = 355$ 
					               and thus $\alpha_3\leq 118$. On the other hand, $\alpha_3\geq \binom{10}{3}=120$, a contradiction.																								
					 \item $n=10$. Since $\alpha_4=\binom{9}{4}=126$, from \eqref{cond2} it follows that
					               $\alpha_3\leq 83$. On the other hand, $\alpha_3\geq \binom{9}{3}=84$, a contradiction.																	
					 \item $n=9$. We have $\alpha_4=70$. From \eqref{cond2} it follows that $\alpha_3 \leq 58$. 
					              Also $\alpha_2\geq \binom{8}{2}=28$. Therefore
												$$\beta_5^6 \leq f_5(8)+3\cdot 58 - 4\cdot 28 + 45 - 6 = 17<\binom{7}{5}=21.$$
					 \item $n=8$. From \eqref{necesar} it follows that $\alpha_5\geq \binom{6}{5}+5\cdot 8-23 = 23 > \binom{7}{5}$
					              and therefore $\alpha_4>\binom{7}{4}$, a contradiction.
					 \end{itemize}
					
\item[(b)] $\alpha_4=\binom{n_4}{4}+\binom{n_3}{3}$ with $3\leq n_3 < n_4 < n$. From \eqref{necesar} it follows that $n_4\geq n-2$.
           If $n\geq 13$ then
           \begin{align*} 
					 & f_5(n)-(f_5(n_4)+f_4(n_3))\geq f_5(n)-(f_5(n-1)+f_4(n-2))=\\
					 & = \frac{1}{6}(n-2)(n-3)(n-10)\geq 0,
					 \end{align*}
					 and we are done. We consider several subcases:
					 \begin{itemize}
					 \item $n=12$. If $n_4=11$ then $f_5(n_4)+f_4(n_3)\leq f_5(n_4)=f_5(11)=f_5(12)$ and we are done.
					       Assume that $n_4=10$. From \eqref{necesar} it follows that $\alpha_5\geq \binom{10}{5}+37$. Since $\binom{7}{4}<37\leq \binom{8}{4}$,
								 if follow that $n_3\in \{8,9\}$ and thus $f_5(n_4)+f_4(n_3) = -168-42=-210<-198-f_5(12)$.
											
					 \item $n=11$. From \eqref{cond1} we have $\alpha_3\geq 80$ and $\alpha_2\geq 40$.
					       Assume that $n_4=10$. If $n_3\geq 7$ then $f_5(n_4)+f_3(n_3)\leq -168-35=-203<-198=f_5(11)$, and we are done.
								 If $n_3\leq 6$ then 														
								 $\alpha_4\leq 230$ and thus, using \eqref{cond2}, $\alpha_3\leq 125$. Therefore								
								 $$\beta_5^6 \leq f_5(10)+f_4(3)+3\cdot 125 - 4\cdot 40 + 5\cdot 11 - 6 = 94 < \binom{9}{5}=126.$$								
								 If $n_4=9$ then $\alpha_4\leq \binom{9}{4}+\binom{8}{3}$ and thus $\alpha_3\leq 109$. Therefore
								 $$\beta_5^6 \leq f_5(9)+f_4(3)+3\cdot 109 - 4\cdot 40 + 5\cdot 11 - 6 =  90 < \binom{9}{5}=126.$$
								 
					 \item $n=10$. We have $\alpha_3\geq 60$ and $\alpha_2\geq 35$.					
					       If $n_4=9$ and $n_3=8$ then $f_5(n_4)+f_4(n_3)=-168=f_5(10)$ and we are done. If $n_4=9$ and $4\leq n_3\leq 7$
								 then $\alpha_4\leq \binom{9}{4}+\binom{7}{3}$ and $\alpha_3\leq 95$. Therefore
								 $$\beta_5^6 \leq f_5(9)+f_4(4) + 3\cdot 95 - 4\cdot 35+5\cdot 10 - 6=56 = \binom{8}{5}.$$
								 If $n_4=9$ and $n_3=3$ then $\alpha_4=127$ and $\alpha_3\leq 84$. Therefore
								 $$\beta_5^6 \leq f_5(9)+f_4(3) + 3\cdot 84 - 4\cdot 35+5\cdot 10 - 6=28 < \binom{8}{5}.$$
                 Assume $n_4=8$. Since $\alpha_5\geq \binom{8}{5}+27$ and $\binom{6}{4}<27$ it follows that $n_3=7$.
								 Therefore
                 $\alpha_4 = \binom{8}{4}+\binom{7}{3}=105$ and $\alpha_3\leq 76$. It follows that
								 $$\beta_5^6 \leq f_5(8)+f_4(3) + 3\cdot 76 - 4\cdot 35 + 5\cdot 10 - 6 = 45<\binom{8}{5}=56.$$

					 \item  $n=9$. We have $\alpha_3\geq 50$ and $\alpha_2\geq 30$. 
					        Since $\alpha_5\geq \binom{7}{5}+ 22$ and $22>\binom{6}{4}=15$, it follows that $n_4=8$.                  
					        If $n_3\geq 6$ then $\alpha_4\leq 105$ and $\alpha_3\leq 70$. Therefore
									$$\beta_5^6 \leq f_5(8)+f_4(6) + 3\cdot 70 - 4\cdot 30 + 5\cdot 9 - 6 = 20<\binom{7}{5}=21.$$
									If $n_3\leq 5$ then $\alpha_4\leq 80$ and $\alpha_3\leq 61$. Therefore
									$$\beta_5^6 \leq f_5(8)+f_4(3) + 3\cdot 61 - 4\cdot 30 +5\cdot 9 - 6 = 16,$$
									and we are done. 									
					 \item  $n=8$. We have $\alpha_3\geq 40$ and $\alpha_2\geq 25$.
					        Since $\alpha_5\geq \binom{6}{5}+17=23=\binom{7}{5}+\binom{4}{4}+\binom{3}{3}$, it follows
					        that $\alpha_4\geq \binom{7}{4}+\binom{5}{3} = 45$, i.e. $\alpha_4\in\{5,6\}$.										       
					        If $n_3=6$ then $\alpha_4 =55$ and thus $\alpha_3\leq 46$. Hence
									$$\beta_5^6 \leq f_5(7)+f_4(6) +3\cdot 46 - 4\cdot 25 +5\cdot 8 - 6 =-2<0, $$
									a contradiction. If $n_3=5$ then $\alpha_4=45$ and thus $\alpha_3\leq 43$. Hence
									$$\beta_5^6 \leq f_5(7)+f_4(5) +3\cdot 43 - 4\cdot 25 +5\cdot 8 - 6 =-1<0, $$
									again a contradiction.
					 \end{itemize}
\item[(c)] $\alpha_4=\binom{n_4}{4}+\binom{n_3}{3}+\binom{n_2}{2}$ with $2\leq n_2<n_3<n_4<n$. From \eqref{necesar} it follows that $n_4\geq n-2$.
           If $n\geq 13$ then
           \begin{align*}
           & f_5(n)-(f_5(n_4)+f_4(n_3)+f_3(n_2))\geq \\
					 &\geq f_5(n)-(f_5(n-1)+f_4(n-2)+f_3(n-3))=\frac{1}{2}(n-3)(n-8)\geq 0,
					 \end{align*}
					 and we are done. 
           We consider several subcases:   
					 \begin{itemize}					 
					 \item $n=12$. We have $\alpha_3\geq 80$ and $\alpha_2\geq 45$. If $n_4=11$ then 
					       $$f_5(n_4)+f_4(n_3)+f_3(n_2)\leq f_5(n_4)-4 = f_5(11)-4 = -202< -198=f_5(12),$$
								 and we are done. Using a similar argument as in the case (b), if $n_4=10$ then $\alpha_4\geq \binom{10}{4}+\binom{7}{3}+\binom{4}{2}$.
								 It follows that $$f_5(n_4)+f_4(n_3)+f_3(n_2)\leq -43-168 = - 211 < -198=f(12).$$
					 \item $n=11$. We have $\alpha_3\geq 70$ and $\alpha_2\geq 40$. If $n_4=10$ and $n_3\geq 7$ then we are done as in the case (b).
					       If $n_4=10$ and $n_3\leq 6$ then $\alpha_4\leq 240$. From \eqref{cond2} it follows that $\alpha_3\leq 128$. Therefore
								  $$\beta_5^6 \leq f_5(10)+f_4(3)+f_3(2)+3\cdot 128 - 4\cdot 40 + 5\cdot 11 - 6 = 101 < \binom{9}{5}=126.$$
								  Since $\alpha_5\geq \binom{9}{5}+37$ it follows that $\alpha_4\geq \binom{9}{4}+\binom{7}{3}+\binom{4}{2}$.
									On the other hand, $\alpha_4\leq \binom{9}{4}+\binom{8}{3}+\binom{7}{2}=203$ and therefore $\alpha_3\leq 116$. Therefore
									$$\beta_5^6 \leq f_5(9) - 43 + 3\cdot 116 - 4\cdot 40 +5\cdot 11 - 6 = 68\leq \binom{9}{5}=126.$$
									
					 \item $n=10$. We have $\alpha_3\geq 60$ and $\alpha_2\geq 35$. Assume $n_4=9$.
					       If $n_3=8$ then we are done as in the case (b). If $6\leq n_3\leq 7$ then
								 $\alpha_4\leq \binom{9}{4}+\binom{7}{3}+\binom{6}{2}=176$ and thus $\alpha_3\leq 100$. Therefore
								 $$\beta_5^6\leq f_5(9)+f_4(6)+3\cdot 100 - 4\cdot 35 + 5\cdot 10 - 6 = 51<\binom{8}{5}=56.$$
								 If $n_3\leq 5$ then $\alpha_4\leq \binom{9}{4}+\binom{5}{3}+\binom{4}{2}=142$ and thus $\alpha_3\leq 89$.
								 Therefore
								 $$\beta_5^6\leq f_5(9)+f_4(3)+f_3(2)+ 3\cdot 89 - 4\cdot 35 + 5\cdot 10 - 6 = 41<\binom{8}{5}=56.$$
								 Now, assume $n_4=8$. Since $\alpha_5\geq \binom{8}{5}+27$, it follows that $\alpha_4\geq \binom{8}{4}+\binom{6}{3}+\binom{5}{2}$.
								 Also, since $\alpha_4\leq \binom{8}{4}+\binom{7}{3}+\binom{6}{2}=120$, it follows that $\alpha_3\leq 81$. Therefore
								 $$\beta_5^6\leq f_5(8)+f_4(6)+f_3(5) + 3\cdot 81 - 4\cdot 35 + 5\cdot 10 - 6 = 40<\binom{8}{5}.$$
								
					 \item  $n=9$. Assume $n_4=8$.					
					        If $n_3\geq 6$ then $\alpha_4\leq \binom{8}{4}+\binom{7}{3}+\binom{6}{2}=120$ and $\alpha_3\leq 75$.
									Also, $\alpha_2\geq \binom{8}{2}+\binom{6}{1}+1=35$. Therefore
									$$\beta_5^6 \leq f_5(8)+f_4(6)+f_3(2) + 3\cdot 75 - 4\cdot 35 + 5\cdot 9 - 6 = 13<\binom{7}{5}=21.$$									
									If $n_3\leq 5$ then $\alpha_4\leq \binom{8}{4}+\binom{5}{3}+\binom{4}{2}=86$ and $\alpha_3\leq 63$. 
									Also, from \eqref{cond1} we have $\alpha_2\geq 30$.
									Therefore
									$$\beta_5^6 \leq f_5(8)+f_4(3)+f_3(2) + 3\cdot 63 - 4\cdot 30 +5\cdot 9 - 6 = 20<21,$$
									and we are done. Assume $n_4=7$. Since $\alpha_5\geq \binom{7}{5}+ 22$ it follows that $n_3=6$ and $n_2=5$.
									Therefore $\alpha_4=65$ and thus $\alpha_3\leq 56$. Hence
									$$\beta_5^6 \leq f_5(7)+f_4(6)+f_3(5) + 3\cdot 56 - 4\cdot 30 +5\cdot 9 - 6 = 3<21.$$
					
					 \item  $n=8$. We have $\alpha_3\geq 40$ and $\alpha_2\geq 25$. Since $\alpha_5\geq \binom{7}{5}+\binom{4}{4}+\binom{3}{3}$
					        it follows that $n_4=7$ and, either $n_3=4$ and $n_2=3$, either $n_3\in \{5,6\}$.
									If $n_3=6$ and $n_2\geq 3$ then $\alpha_4 \leq  \binom{7}{4}+\binom{6}{3}+\binom{5}{2}= 65$. It follows that $\alpha_3\leq 50$.
									Hence
									$$\beta_5^6 \leq f_5(7)+f_4(6)+f_3(3)+3\cdot 50 - 4\cdot 25 +5\cdot 8 - 6 =5 < \binom{6}{5}=6.$$
									If $n_3=6$ and $n_2=2$ then $\alpha_4=56$ and $\alpha_3\leq 47$ and we are done using a similar argument as above.
									
									If $n_3=5$ then $\alpha_4 \leq  \binom{7}{4} + \binom{5}{3}+\binom{4}{2} =51$ and $\alpha_3\leq 45$. It follows that
									$$\beta_5^6 \leq f_5(7) + f_4(5)+f_3(2)+  3\cdot 45-4\cdot 25 + 5\cdot 8 - 6 = 3 <6.$$
									If $n_3=4$ and $n_2=3$ then $\alpha_4=\binom{7}{4}+\binom{4}{3}+\binom{3}{2}=42$ and $\alpha_3\leq 42$. It follows that
									$$\beta_5^6 \leq f_5(7) + f_4(4)+f_3(3)+  3\cdot 42-4\cdot 25 + 5\cdot 8 - 6 = -1<0,$$
									a contradiction.									
					 \end{itemize}
\item[(d)] $\alpha_4=\binom{n_4}{4}+\binom{n_3}{3}+\binom{n_2}{2}+\binom{n_1}{1}$ with $1\leq n_1<n_2<n_3<n_4<n$. From \eqref{necesar} it follows that $n_4\geq n-2$.
           If $n\geq 13$ then
           \begin{align*}
					 & f_5(n)-(f_5(n_4)+f_4(n_3)+f_3(n_2)+f_2(n_1))\geq \\
					 & f_5(n)-(f_5(n-1)+f_4(n-2)+f_3(n-3)+f_2(n-4))=n-6\geq 0,
					 \end{align*}
					 and we are done. We consider several subcases:   
					 \begin{itemize}					 
					 \item $n=12$. We have $\alpha_3\geq 80$ and $\alpha_2\geq 45$. If $n_4=11$ then 
					       $$f_5(n_4)+f_4(n_3)+f_3(n_2)+f_4(n_1)\leq f_5(n_4)-6 = f_5(11)-6 = -204< -198=f_5(12),$$
								 and we are done. Using a similar argument as in the case (b), if $n_4=10$ then $\alpha_4\geq \binom{10}{4}+\binom{7}{3}+\binom{3}{2}+\binom{2}{1}$.
								 It follows that $f_5(n_4)+f_4(n_3)+f_3(n_2)+f_2(n_1)\leq -39-168 = - 207 < -198=f(12)$, as required.								
					 \item $n=11$. We have $\alpha_3\geq 70$ and $\alpha_2\geq 40$. If $n_4=10$ and $n_3\geq 7$ then we are done as in the case (b).
					       If $n_4=10$ and $n_3\leq 6$ then $\alpha_4\leq 244$. From \eqref{cond2} it follows that $\alpha_3\leq 129$. Therefore
								  $$\beta_5^6 \leq f_5(10)+f_4(3)+f_3(2)+f_2(1)+3\cdot 129 - 4\cdot 40 + 5\cdot 11 - 6 = 102 < \binom{9}{5}=126.$$
								  Since $\alpha_5\geq \binom{9}{5}+37$ it follows that $\alpha_4\geq \binom{9}{4}+\binom{7}{3}+\binom{3}{2}+\binom{2}{1}$.
									On the other hand, $\alpha_4\leq \binom{10}{4}-1=209$ and therefore $\alpha_3\leq 118$. Therefore
									$$\beta_5^6 \leq f_5(9) - 43 + 3\cdot 118 - 4\cdot 40 +5\cdot 11 - 6 = 74\leq \binom{9}{5}=126.$$
					
					 \item $n=10$. We have $\alpha_3\geq 60$ and $\alpha_2\geq 35$. Assume $n_4=9$.
					       If $n_3=8$ then we are done as in the case (b). If $6\leq n_3\leq 7$ then
								 $\alpha_4\leq \binom{9}{4}+\binom{7}{3}+\binom{6}{2}+\binom{5}{1}=181$ and thus $\alpha_3\leq 100+2$. Therefore
								 $$\beta_5^6\leq f_5(9)+f_4(6)+f_3(2)+f_3(1)+3\cdot 102 - 4\cdot 35 + 5\cdot 10 - 6 = 55 < \binom{8}{5}=56.$$
								 If $n_3\leq 5$ then $\alpha_4\leq \binom{9}{4}+\binom{5}{3}+\binom{4}{2}+\binom{3}{1}=145$ and thus $\alpha_3\leq 90$.
								 Therefore
								 $$\beta_5^6\leq f_5(9)+f_4(3)+f_3(2)+f_2(1)+ 3\cdot 90 - 4\cdot 35 + 5\cdot 10 - 6 = 42<\binom{8}{5}=56.$$
								 Now, assume $n_4=8$. Since $\alpha_5\geq \binom{8}{5}+27$, it follows that $\alpha_4\geq \binom{8}{4}+\binom{6}{3}+\binom{5}{2}+\binom{2}{1}$.
								 Also, since $\alpha_4\leq \binom{8}{4}+\binom{7}{3}+\binom{6}{2}+\binom{5}{1}=125$, it follows that $\alpha_3\leq 83$. Therefore
								 $$\beta_5^6\leq f_5(8)+f_4(6)+f_3(5)+f_2(2) + 3\cdot 83 - 4\cdot 35 + 5\cdot 10 - 6 = 43<\binom{8}{5}.$$
								
					 \item  $n=9$. We have $\alpha_3\geq 50$ and $\alpha_2\geq 30$.  
					        If $n_3\geq 6$ then $\alpha_4\leq \binom{9}{4}-1=125$ and $\alpha_3\leq 76$.
									Also, $\alpha_2\geq \binom{8}{2}+\binom{6}{1}+1=35$. Therefore
									$$\beta_5^6 \leq f_5(8)+f_4(6)+f_3(2) + f_2(1) + 3\cdot 76 - 4\cdot 35 + 5\cdot 9 - 6 = 15<\binom{7}{5}=21.$$									
									If $n_3\leq 5$ then $\alpha_4\leq \binom{8}{4}+\binom{5}{3}+\binom{4}{2}+\binom{3}{1}=89$ and $\alpha_3\leq 64$. 
									Also, from \eqref{cond1} we have $\alpha_2\geq 30$.
									Therefore
									$$\beta_5^6 \leq f_5(8)+f_4(3)+f_3(2)+f_2(1) + 3\cdot 64 - 4\cdot 30 +5\cdot 9 - 6 = 21,$$
									and we are done. Assume $n_4=7$. Since $\alpha_5\geq \binom{7}{5}+ 22$ it follows that $n_3=6$ and $n_2\geq 4$.
									On the other hand, $\alpha_4\leq \binom{8}{4}-1=69$ and thus $\alpha_3\leq 58$. Hence
									$$\beta_5^6 \leq f_5(7)+f_4(6)+f_3(4)+f_2(1) + 3\cdot 58 - 4\cdot 30 +5\cdot 9 - 6 = 9<21.$$

					 \item  $n=8$. We have $\alpha_3\geq 40$ and $\alpha_2\geq 25$. As in the case (c), we have $n_4=7$ and, either $n_3=4$ and $n_2=3$, 
					        either $n_3\in \{5,6\}$. Assume $n_3=6$ and $n_2\geq 3$.
									Since $\alpha_4\leq \binom{8}{4}-1=69$, it follows that 
									
									If $n_3=6$ and $n_2\geq 3$ then $\alpha_4 \leq  \binom{7}{4}+\binom{6}{3}+\binom{5}{2}= 65$. It follows that $\alpha_3\leq 51$.
									Hence
									$$\beta_5^6 \leq f_5(7)+f_4(6)+f_3(3)+f_2(1)+3\cdot 51 - 4\cdot 25 +5\cdot 8 - 6 =6 =\binom{6}{5}.$$
									If $n_3=6$ and $n_2=2$ then $\alpha_4=57$ and $\alpha_3\leq 47$ and we are done using a similar argument as above.
									
									If $n_3=5$ then $\alpha_4 \leq  \binom{7}{4} + \binom{5}{3}+\binom{4}{2} + \binom{3}{1} =54$ and $\alpha_3\leq 46$. It follows that
									$$\beta_5^6 \leq f_5(7) + f_4(5)+f_3(2)+f_2(1)+  3\cdot 46-4\cdot 25 + 5\cdot 8 - 6 = 4 <6.$$
									If $n_3=4$ and $n_2=3$ then $\alpha_4\leq \binom{7}{4}+\binom{4}{3}+\binom{3}{2}+\binom{2}{1}=44$ and $\alpha_3\leq 43$. It follows that
									$$\beta_5^6 \leq f_5(7) + f_4(4)+f_3(3)+f_2(1)+  3\cdot 44-4\cdot 25 + 5\cdot 8 - 6 = 0.$$									
					 \end{itemize}
\end{enumerate}
Hence, the proof is complete.
\end{proof}

\begin{lema}\label{b66}
With the above notations, we have that $\beta_6^6\leq \binom{n-1}{6}$.
\end{lema}

\begin{proof}
Since $\beta_6^6=\alpha_6-\beta_5^5$ and $\beta_5^5\geq 0$, the conclusion follows immediately if $\alpha_6\leq \binom{n-1}{5}$.
 Hence, we may assume that $$\alpha_6\geq \binom{n-1}{6}+1 = \binom{n-1}{6}+\binom{5}{5}.$$
 From Lemma \ref{cord} it follows that 
$$\alpha_j\geq \binom{n-1}{j}+\binom{5}{j-1}\text{ for all }j\in\{2,3,4,5\}.$$

We denote $f_k(x)=\binom{x}{k}-\binom{x}{k-1}$ for all $2\leq k\leq 6$. We have the following table of values:

\begin{table}[tbh]
\begin{tabular}{|l|l|l|l|l|l|l|l|l|l|l|l|}
\hline
$x$      & 1  & 2  & 3  & 4  & 5  & 6  & 7   & 8   & 9  & 10 & 11  \\ \hline
$f_6(x)$ & 0  & 0  & 0  & 0  & -1 & -5 & -14 & -28 & -42& -42& 0   \\ \hline
$f_5(x)$ & 0  & 0  & 0  & -1 & -4 & -9 & -14 & -14 & 0  & 42 & 132 \\ \hline
$f_4(x)$ & 0  & 0  & -1 & -3 & -5 & -5 & 0   & 14  & 42 & 90 & 165 \\ \hline
$f_3(x)$ & 0  & -1 & -2 & -2 & 0  & 5  & 14  & 28  & 48 & 75 & 110 \\ \hline
$f_2(x)$ & -1 & -1 & 0  & 2  & 5  & 9  & 14  & 20  & 27 & 35 & 44  \\ \hline
\end{tabular}
\end{table}

Since $n\geq 8$, from the above table, it follows that
$$f_4(n)-(\alpha_4-\alpha_3) \geq f_4(n)-(f_4(n-1)+f_3(n-2)+f_2(n-3))\geq n-4,$$
and thus $\alpha_4-\alpha_3\leq \binom{n}{4}-\binom{n}{3}$.
In particular, we get
$$\beta_4^4 = \alpha_4-\alpha_3+\alpha_2-n+1 \leq \binom{n-1}{4}.$$
Hence, in order to prove that $\beta_6^6\leq \binom{n-1}{6}$ it suffices to show that 
$$\alpha_6-\alpha_5\leq f_6(n).$$
Note that, if $\alpha_5 = \binom{n-1}{5}+\binom{n_4}{4}$ then $\alpha_6\leq \binom{n-1}{6}+\binom{n_4}{5}$ and therefore
$$\alpha_6-\alpha_5\leq f_6(n-1)+f_5(n_4).$$ 
Similarly, if $\alpha_5 = \binom{n-1}{5}+\binom{n_4}{4}+\binom{n_3}{3}$ then
$\alpha_6-\alpha_5\leq f_6(n-1)+f_5(n_4)+f_4(n_3)$ etc.

Let $g_k(x)=\binom{x}{k}-2\binom{x}{k-1}$, for $k=2,3$. We have the table:

\begin{table}[tbh]
\begin{tabular}{|l|l|l|l|l|l|l|l|l|l|}
\hline
x        & 1  & 2  & 3  & 4  & 5   & 6   & 7  & 8  & 9  \\ \hline
$g_3(x)$ & 0  & -2 & -5 & -8 & -10 & -10 & -7 & 0  & 12 \\ \hline
$g_2(x)$ & -2 & -3 & -3 & -2 & 0   & 3   & 7  & 12 & 18 \\ \hline
\end{tabular}
\end{table}

Since $\alpha_2\geq \binom{n-1}{2}+\binom{5}{1}$ and $n\geq 8$, from the above table it follows that
$$\alpha_3-2\alpha_2 \leq \max\{g_3(n-1)+g_2(n-2), g_3(n)\}=g_3(n)$$
Since $\beta_5^6\geq 0$ it follows that
$$2(\alpha_4-\alpha_3+\alpha_2) \leq \alpha_5 + (\alpha_3-2\alpha_2) + 5n-6.$$
From the above relations, it follows that
\begin{equation}\label{25}
\beta_4^4 = \alpha_4-\alpha_3+\alpha_2-n+1 \leq \frac{1}{2}(\alpha_5+3n) + \frac{1}{12}n(n-1)(n-8)-2.
\end{equation}

We consider several cases:
\begin{enumerate}
\item $n\geq 11$. From the table of values of $f_k(x)$'s, we deduce that \small
      $$f_6(n)-(\alpha_6-\alpha_5) \geq f_6(n)-(f_6(n-1)+f_5(n-2)+f_4(n-3)+f_3(n-4)+f_2(n-5))=n-6\geq 5,$$ \normalsize
		  and we are done.
\item $n=10$. Since $\alpha_5\geq \binom{9}{5}+\binom{5}{4}$, from the table of values of $f_k(x)$'s, we deduce
      than $$\alpha_6-\alpha_5\leq f_6(9)-4=-46<-42=f_6(10).$$
\item $n=9$. If $\alpha_5\geq \binom{8}{5}+\binom{6}{4}+\binom{4}{3}+\binom{3}{2}$ then, from the table of values of $f_k(x)$'s, we deduce
      $$\alpha_6-\alpha_5\leq f_6(8)-14=-42=f_6(9).$$
			Now, assume that $\alpha_5\leq \binom{8}{5}+\binom{6}{4}+\binom{4}{3}+\binom{2}{2}=76$. From \eqref{25} it follows that
		  $\beta_4^4\leq 55$. Also, $\alpha_5\geq \binom{8}{5} + \binom{5}{4}$. Therefore
			$$\beta_6^6 \leq \alpha_6-\alpha_5+55 \leq f_6(8)+f_5(5)+55=-32+55=23 \leq \binom{9-1}{6}=28.$$			
\item $n=8$. If $\alpha_5\geq \binom{7}{5}+\binom{6}{4}+\binom{4}{3}+\binom{2}{2}+\binom{1}{1}=42$ then, from the table of values of $f_k(x)$'s,
      we deduce that $\alpha_6-\alpha_5\leq -28=f_6(8)$ and we are done. 			
			Assume that $$38=\binom{7}{5}+\binom{6}{4}+\binom{3}{3}+\binom{2}{2}\leq \alpha_5\leq 41.$$ 
			From \eqref{25} it follows that	$\beta_4^4\leq 32$. Therefore
			$$\beta_6^6 \leq \alpha_6-\alpha_5+32 \leq f_6(7)+f_5(6)-3+32=-25+32=7 =\binom{8-1}{6}.$$
			If $\alpha_5\leq 37$ then, from \eqref{25} it follows that $\beta_4^5\leq 28$. Therefore, if $\alpha_5\geq \binom{7}{5}+\binom{5}{4}+\binom{4}{3}=30$ then
			$$\beta_6^6 \leq \alpha_6-\alpha_5+28 \leq -21+28 = 7.$$
			Now, if $\alpha_5\leq 29$ then  $\beta_4^4\leq 24$. It follows that
			$$\beta_6^6 \leq \alpha_6-\alpha_5+24 \leq f_6(7)+f_5(5)+24 = -18+24=6.$$
\end{enumerate}
Hence, the proof is complete.
\end{proof}

\begin{teor}\label{main2}
Let $I\subset S$ be a squarefree monomial ideal with $\hdepth(S/I)=6$. Then $\hdepth(I)\geq 6$.
\end{teor}

\begin{proof}
The conclusion follows from Lemma \ref{lem}, Lemma \ref{b3q} (the case $q=6$), Lemma \ref{b46}, Lemma \ref{b56} and Lemma \ref{b66}.
\end{proof}


\begin{lema}\label{79}
If $I\subset K[x_1,\ldots,x_9]$ is a squarefree monomial ideal with $\qdepth(S/I)=7$, 
then $\beta_k^7 \leq k+1$ for all $3\leq k\leq 7$.
\end{lema}

\begin{proof}
The case $k=3$ follows from Lemma \ref{b3q}. Since $\qdepth(S/I)=7$ we have:
\begin{align*}
& \beta_2^7 = \alpha_2-33\geq 0,\\
& \beta_3^7 = \alpha_3-5\alpha_2+100 \geq 0, \\
& \beta_4^7 = \alpha_4-4\alpha_3+10\alpha_2 - 145 \geq 0, \\
& \beta_5^7 = \alpha_5-3\alpha_4+ 6\alpha_3 - 10\alpha_2 + 114 \geq 0, \\
& \beta_6^7 = \alpha_6-2\alpha_5+3\alpha_4-4\alpha_3+5\alpha_2 - 47 \geq 0,\\
& \beta_7^7 = \alpha_7-\alpha_6+\alpha_5-\alpha_4+\alpha_3-\alpha_2+8\geq 0.
\end{align*}
On the other hand, from \eqref{betak} we can easily deduce that
$$\alpha_k=\sum_{j=0}^k \binom{7-j}{k-j} \beta_j^7\text{ for all }0\leq k\leq 7.$$
Since $\beta_0^7=1$ and $\beta_1^7=2$ it follows that 
 \begin{align*}
 & \alpha_2\geq \binom{7}{2}+2\cdot \binom{6}{1} = 33,\; 
  \alpha_3\geq \binom{7}{3}+2\cdot \binom{6}{2} = 65,\; 
  \alpha_4\geq \binom{7}{4}+2\cdot \binom{6}{3} = 75, \\
 & \alpha_5\geq \binom{7}{5}+2\cdot \binom{6}{4} = 51,\;
  \alpha_6\geq \binom{7}{6}+2\cdot \binom{6}{5} = 19,\;
  \alpha_7\geq \binom{7}{7}+2\cdot \binom{6}{6} = 3.
 \end{align*}
Moreover, since $\beta_j^7 \geq 0$ for all $2\leq j\leq 7$, in order to deduce that $\beta_4^7\leq 5$ we
can assume 
$$\alpha_4\geq 81=\binom{8}{4}+\binom{5}{3}+\binom{2}{2}.$$ 
Indeed, if $\alpha_4\leq 80$ then 
$$\beta_4^7 = \alpha_4 - \binom{7}{3}\beta_0^7 - \binom{6}{2}\beta_1^7 - \binom{5}{1}\beta_2^7 \leq \alpha_4-75\leq 5,$$
as required. Similarly, in order to prove that $\beta_5^7\leq 6$ we can assume that 
$$\alpha_5\geq 58=\binom{8}{5}+\binom{4}{4}+\binom{3}{3}.$$
Also, in order to prove that $\beta_6^7\leq 7$ we can assume that $$\alpha_6\geq 27=\binom{8}{6}-1,$$ and, finally, in order to
prove that $\beta_7^7\leq 8$ we can assume that 
$$\alpha_7\geq 12=\binom{8}{7}+\binom{6}{6}+\binom{5}{5}+\binom{4}{4}+\binom{3}{3}.$$
We have several cases:
\begin{enumerate}
\item[(a)] $k=4$. Since $\beta_3^7\geq 0$ it follows that $\alpha_2\leq \frac{1}{5}\alpha_3+20.$
For $2\leq k\leq 4$ we let $f_k(x)=\binom{x}{k}-4\cdot\binom{x}{k-1}$. We have the table:

\begin{center}
\begin{table}[tbh]
\begin{tabular}{|l|l|l|l|l|l|l|l|l|l|}
\hline
$x$       & 1  & 2  & 3   & 4   & 5   & 6   & 7    & 8    & 9    \\ \hline
$f_4(x)$ & 0  & 0  & -4  & -15 & -35 & -65 & -105 & -154 & -210 \\ \hline
$f_3(x)$ & 0  & -4 & -11 & -20 & -30 & -40 & -49  &     &    \\ \hline
$f_2(x)$ & -4 & -7 & -9  & -10 & -10 & -9  &     &     &    \\ \hline
\end{tabular}
\end{table}
\end{center}

Since $\beta_4^7 = \alpha_4-4\alpha_3+10\alpha_2 - 145$, it is enough to show that 
$$\alpha_4-4\alpha_3+10(\alpha_2-36)\leq f_4(9)=-210,$$ 
in order to obtain $\beta_4^7\leq 5$. We consider the subcases:
\begin{enumerate}
\item If $\alpha_3\geq 80=\binom{8}{3}+\binom{7}{2}+\binom{3}{1}$ then $\alpha_4-4\alpha_3\leq f_4(8)+f_3(7)+f_2(3)=-212$ and we are done.
\item If $79\geq \alpha_3\geq 75 = \binom{8}{3} +\binom{6}{2}+\binom{4}{1}$ then $\alpha_2\leq 35$ and therefore
      $$\alpha_4-4\alpha_3-10 \leq -154-49-10 = -213.$$
\item If $74\geq \alpha_3\geq 70 = \binom{8}{3} +\binom{5}{2}+\binom{4}{1}$ then $\alpha_2\leq 34$ and therefore
      $$\alpha_4-4\alpha_3-20 \leq -154-40-20 = -214.$$
\item If $69\geq \alpha_3\geq 65 = \binom{8}{3} +\binom{4}{2}+\binom{3}{1}$ then $\alpha_2 = 33$ and therefore
      $$\alpha_4-4\alpha_3-30 \leq -154-29-30 = -213.$$
\end{enumerate}

\item[(b)] $k=5$. For $2\leq k\leq 5$ we let $f_k(x)=\binom{x}{k}-3\cdot\binom{x}{k-1}$. We have the table:


\begin{center}
\begin{table}[tbh]
\begin{tabular}{|l|l|l|l|l|l|l|l|l|l|}
\hline
$x$       & 1  & 2  & 3   & 4   & 5   & 6   & 7    & 8    & 9    \\ \hline
$f_5(x)$  & 0  & 0  & 0   & -3  & -14 & -39 & -84  & -154 & -252 \\ \hline
$f_4(x)$  & 0  & 0  & -3  & -11 & -25 & -45 & -70  &      &   \\ \hline
$f_3(x)$  & 0  & -3 & -8  & -14 & -20 & -25 &      &      &    \\ \hline
$f_2(x)$  & -3 & -5 & -6  & -6  & -5  &     &      &      &     \\ \hline
\end{tabular}
\end{table}
\end{center}

Since $\alpha_3\geq 5\alpha_2-100$, we have also the following table:

\begin{center}
\begin{table}[tbh]
\begin{tabular}{|l|l|l|l|l|}
\hline
$\alpha_2$                   & 33 & 34 & 35  & 36  \\ \hline
$\min(\alpha_3)$             & 65 & 70 & 75  & 80  \\ \hline
$\max(\alpha_3)$             & 66 & 71 & 77  & 84  \\ \hline
$\max(6\alpha_3-10\alpha_2)$ & 66 & 86 & 112 & 144 \\ \hline
\end{tabular}
\end{table}
\end{center}

Now, we consider the following subcases:
\begin{enumerate}
\item[(i)] $\alpha_2=36$. From this, we note that $6\alpha_3-10\alpha_2\leq 144$, so in order to obtain $\beta_5^7\leq 6$, it is enough to show that 
$$ \alpha_5-3\alpha_4 + 144 + 114 \leq 6 \Leftrightarrow \alpha_5-3\alpha_4 \leq - 252.$$
If $\alpha_4\geq \binom{8}{4}+\binom{7}{3}+\binom{6}{2}+\binom{1}{1} = 120$ then
$$\alpha_5-3\alpha_4 \leq f_5(8)+f_4(7)+f_3(6)+f_2(1) = -252,$$
and we are done. Hence, we may assume in the following that $75\leq \alpha_4\leq 119$.
Since 
$$\beta_4^7 = \alpha_4-4\alpha_3 + 360-145 = \alpha_4-4\alpha_3 + 215 \geq 0,$$
it follows that $\alpha_3\leq \frac{1}{4}(\alpha_4+215)$. Also, in order to prove that $\beta_5^7\leq 6$
it is enough to show that
\begin{equation}\label{pohta}
 \alpha_5-3\alpha_4 \leq 252-6\alpha_3. 
\end{equation}
Since $\alpha_4\leq 119$ then $\alpha_3\leq 83$ and $252-6\alpha_3\geq -246$.
If $\alpha_4\geq \binom{8}{4}+\binom{7}{3}+\binom{5}{2}+\binom{1}{1}=116$ then \eqref{pohta} is satisfied.
\begin{itemize}
\item Now, assume $\alpha_4\leq 115$. Then $\alpha_3\leq 82$ and $252-6\alpha_3\geq -240$. 
If $\alpha_4\geq \binom{8}{4}+\binom{7}{3}+\binom{4}{2}+\binom{1}{1}=112$ then $\alpha_5-3\alpha_4\leq -241$
and \eqref{pohta} is satisfied.
\item Now, assume $\alpha_4\leq 111$. Then $\alpha_3\leq 81$ and $252-6\alpha_3 \geq -234$. 
If $\alpha_4\geq \binom{8}{4}+\binom{7}{3}+\binom{3}{2}+\binom{1}{1}=109$ then $\alpha_5-3\alpha_4\leq -235$
and \eqref{pohta} is satisfied.
\item Now, assume $\alpha_4\leq 108$. Then $\alpha_3\leq 80$ and $252-6\alpha_3 \geq -228$. 
If $\alpha_4\geq \binom{8}{4}+\binom{7}{3}+\binom{2}{2}+\binom{1}{1}=107$ then $\alpha_5-3\alpha_4\leq -230$
and \eqref{pohta} is satisfied.
\item  Now, assume $\alpha_4=106$. If $\alpha_3\leq 79$ then $\beta_4^7 < 0$, so $\alpha_3=80$. Note that
       $\alpha_5\leq \binom{8}{5}+\binom{7}{4}=91$ and $\alpha_6\leq \binom{8}{6}+\binom{7}{5}=49$. If $\alpha_5=91$
			 then 
			 $$\beta_7^6 \leq 49 - 2\cdot 91 + 3\cdot 106 - 4\cdot 36 + 5\cdot 9 - 6 \leq -2,$$
			 a contradiction. It follows that $\alpha_5\leq 90$ and thus $\alpha_5-3\alpha_4\leq -227$. 
       On the other hand, as $106=\binom{8}{4}+\binom{7}{3}+\binom{2}{2}$ and $f_5(8)+f_4(7)+f_3(2)=-227$,
			 we are done.
\item  Now, assume $\alpha_4=105=\binom{8}{4}+\binom{7}{3}$. As above, we have $\alpha_5\leq 91$ and $\alpha_6\leq 49$.
       If $\alpha_5\geq 89$ then we get again a contradiction, hence we may assume $\alpha_5\leq 88$. 
			 If $\alpha_5=88=\binom{8}{5}+\binom{6}{4}+\binom{5}{3}+\binom{4}{2}+\binom{1}{1}$,
			 it follows that $\alpha_6\leq 43$. From this we get $\beta_6^7<0$, a contradiction. Thus $\alpha_5\leq 87$
			 and $\alpha_5-3\alpha_4\leq -224 = f_5(8)+f_4(7)$ and we are done.
\item Now, assume $\alpha_4\leq 104$. It follows that $\alpha_3\leq 79$, which means $\beta_3^7<0$, a contradiction.
\end{itemize}
\item[(ii)] $\alpha_2=35$. Since $6\alpha_3-10\alpha_2\leq 112$, in order to obtain $\beta_5^7\leq 6$, it is enough to show that 
$$ \alpha_5-3\alpha_4 + 112 + 114 \leq 6 \Leftrightarrow \alpha_5-3\alpha_4 \leq - 220.$$
If $\alpha_5\geq \binom{8}{4}+\binom{6}{3}+\binom{5}{2}+\binom{1}{1}=101$ then 
$$\alpha_5-3\alpha_4\leq f_5(8)+f_4(6)+f_3(5)+f_2(1)=-222,$$
and we are done. So, we may assume $\alpha_4\leq 100$. Since 
$$\beta_4^7 = \alpha_4-4\alpha_3 + 350-145 = \alpha_4-4\alpha_3 + 205 \geq 0,$$
it follows that $\alpha_3\leq \frac{1}{4}(\alpha_4+205)$. Also, in order to prove that $\beta_5^7\leq 6$
it is enough to show that
\begin{equation}\label{pohta2}
 \alpha_5-3\alpha_4 \leq 242-6\alpha_3.
\end{equation}
Since $\alpha_4\leq 100$ it follows that $\alpha_3\leq \frac{305}{4}$ that is $\alpha_3\leq 76$ 
and $242-6\alpha_3 \geq -214$. If $\alpha_4\geq \binom{8}{4}+\binom{6}{3}+\binom{4}{2}+\binom{1}{1}=97$ then
$$\alpha_5-3\alpha_4\leq f_5(8)+f_4(6)+f_3(4)+f_2(1)=-216,$$
and we are done by \eqref{pohta2}.

Assume $\alpha_4\leq 96$. It follows that $\alpha_3=75$ and $242-6\alpha_3 = -208$.
If $\alpha_4 \geq \binom{8}{4}+\binom{6}{3}+\binom{3}{2}+\binom{1}{1}=94$ then 
$$\alpha_5-3\alpha_4\leq f_5(8)+f_4(6)+f_3(3)+f_2(1)=-210,$$
and we are done by \eqref{pohta2}.

Now, assume $\alpha_4\leq 93$. If follows that $\alpha_3\leq 74$, a contradiction.

\item[(ii)] $\alpha_2=34$. Since $6\alpha_3-10\alpha_2\leq 86$, in order to obtain $\beta_5^7\leq 6$, it is enough to show that 
$$ \alpha_5-3\alpha_4 + 86 + 114 \leq 6 \Leftrightarrow \alpha_5-3\alpha_4 \leq - 192.$$
If $\alpha_4\geq \binom{8}{4}+\binom{5}{3}+\binom{3}{2}+\binom{2}{1}=85$, then 
$$\alpha_5-3\alpha_4\leq f_5(8)+f_4(5)+f_3(3)+f_2(2)=-192, $$
and we are done. Hence, we may assume $\alpha_4\leq 84$. Since 
$$\beta_4^7 = \alpha_4-4\alpha_3 + 340-145 = \alpha_4-4\alpha_3 + 195 \geq 0,$$
it follows that $\alpha_3\leq \frac{1}{4}(\alpha_4+195)$. 
Since $\alpha_4\leq 84$ it follows that $\alpha_3\leq 69$, which gives a contradiction.

\item[(iii)] $\alpha_2=33$. Since $6\alpha_3-10\alpha_2\leq 66$, in order to obtain $\beta_5^7\leq 6$, it is enough to show that 
$$ \alpha_5-3\alpha_4 + 66 + 114 \leq 6 \Leftrightarrow \alpha_5-3\alpha_4 \leq - 172.$$
If $\alpha_4\geq \binom{8}{4}+\binom{4}{3}+\binom{3}{2}=77$ then $\alpha_5-3\alpha_4 \leq f_5(8)+f_4(4)+f_3(3)=-173$ and we are done.
It remains to consider the case $\alpha_4\in\{75,76\}$. On the other hand, we have
$$\alpha_3\leq \frac{1}{4}(\alpha_4+185) \leq \frac{261}{4},$$
and therefore $\alpha_3=65$. Also, since $\alpha_4\in\{75,76\}$, we have $\alpha_5\leq \binom{8}{5}+\binom{4}{4}=57$. Hence
$$\beta_5^7 \leq 57 - 3\cdot 75 + 6\cdot 65 - 10\cdot 33 + 15\cdot 9 - 21 = 6=5+1.$$
\end{enumerate}

\item[(c)] $k=6$. For $2\leq k\leq 6$ we let $f_k(x)=\binom{x}{k}-2\cdot\binom{x}{k-1}$. We have the table:

\begin{center}
\begin{table}[tbh]
\begin{tabular}{|l|l|l|l|l|l|l|l|l|l|}
\hline
$x$      & 1  & 2  & 3   & 4   & 5   & 6   & 7    & 8    & 9    \\ \hline
$f_6(x)$ & 0  & 0  & 0   & 0   & -2  & -11 & -35  &  -84 & -168  \\ \hline
$f_5(x)$ & 0  & 0  & 0   & -2  & -9  & -24 &  -49 &      & \\ \hline
$f_4(x)$ & 0  & 0  & -2  & -7  & -15 & -25 &      &      & \\ \hline
$f_3(x)$ & 0  & -2 & -5  & -8  & -10 &     &      &      & \\ \hline
$f_2(x)$ & -2 & -3 & -3  & -2  &     &     &      &      & \\ \hline
\end{tabular}
\end{table}
\end{center}

We claim that 
\begin{equation}\label{3a4}
3\alpha_4-4\alpha_3\leq 3\cdot \binom{9}{4}-4\cdot \binom{9}{3}=42.
\end{equation}
Indeed, assume $\alpha_3<\binom{9}{3}$. Since $\alpha_3\geq 65 >\binom{8}{3}$, we can assume that $\alpha_3=\binom{8}{3}+\binom{n_2}{2}+\binom{n_1}{1}$
for $7\leq n_2\geq 2$ and $0\leq n_1 < n_2$. It follows that $\alpha_4\leq \binom{8}{4}+\binom{n_2}{3}+\binom{n_1}{2}$ and therefore
$$3\alpha_4-4\alpha_3 = -14 + \frac{1}{2}n_2(n_2-1)(n_2-6)+\frac{1}{2}n_1(3n_1-11)\leq -14 + 42 = 28 < 42,$$
hence the claim is true. From \eqref{3a4} it follows that, in order to obtain $\beta_6^7\leq 7$, it is enough to show that 
$$\alpha_6-2\alpha_5\leq -168.$$
If $\alpha_5\geq \binom{8}{5}+\binom{7}{4}+\binom{6}{3}+\binom{4}{2}+\binom{1}{1}=118$ then 
$$\alpha_6-2\alpha_5 \leq f_6(8)+f_5(7)+f_4(6)+f_3(4)+f_2(1)=-168,$$
and we are done. Thus, we can assume $\alpha_5\leq 117$. 

Since $\beta_5^7 = \alpha_5 - 3\alpha_4 + 6\alpha_3 - 10\alpha_2 + 114 \geq 0$ it follows that
\begin{equation}\label{cucucu}
3\alpha_4\leq \alpha_5 +  6\alpha_3 - 10\alpha_2 + 114.
\end{equation}
Also, in order to prove that $\beta_6^7\leq 7$ we need to show that
\begin{equation}\label{cucucu2}
\alpha_6-2\alpha_5 \leq -3\alpha_4+4\alpha_3-5\alpha_2+54.
\end{equation}
We consider several subcases:
\begin{enumerate}
\item[(i)] If $\alpha_2=36$ then $80\leq \alpha_3\leq 84$.       
			\begin{itemize}
			\item $\alpha_3=84$. Since $\beta_4^7\geq 0$ it follows that $\alpha_4\geq 121$.
			      Also $\alpha_4\leq \frac{1}{3}(\alpha_5+258)$, and, in particular $\alpha_5\geq 105$.
					  Moreover, from \eqref{cucucu2}, in order to prove that $\beta_6^7\leq 7$ it is
						enough to show that $$\alpha_6-2\alpha_5 \leq 210-3\alpha_4.$$						
			      Since $\alpha_5\leq 117$ it follows that $\alpha_4\leq 125$. Hence, in order to prove that 
			      $\beta_6^7\leq 7$, it suffice to show that $\alpha_6-2\alpha_5\leq -165$.
		      	If $\alpha_5\geq \binom{8}{5}+\binom{7}{4}+\binom{6}{3}+\binom{3}{2}+\binom{1}{1}=115$ then
		      	$$\alpha_6-2\alpha_5\leq f_6(8)+f_5(4)+f_4(6)+f_3(3)+f_2(1)=-165,$$
			      and we are done. So we can assume $\alpha_5\leq 114$ and thus $\alpha_4\leq 124$.
			      This implies that it is enough to show that $\alpha_6-2\alpha_5\leq -162$. This is true
			      for $\alpha_5\geq 113$. Now, if $\alpha_5\leq 112$ then $\alpha_4\leq 123$ and it is
			      enough to show that $\alpha_6-2\alpha_5\leq -159$. For $\alpha_5\geq 110$ this inequality holds.
						We assume $\alpha_5\leq 109$ which means $\alpha_4\leq 122$ and it is enough to show $\alpha_6-2\alpha_5\leq -156$.
						This is also true for $\alpha_5\geq 107$. Finally, if $\alpha_5\in \{105,106\}$ then $\alpha_4=121$ and
						it is enough to show that $\alpha_6-2\alpha_5\leq -153$, which is true.
						
			\item $\alpha_3=83$. Since $\beta_4^7\geq 0$ it follows that $\alpha_4\geq 117$. On the other hand, since 
			      $\alpha_3=\binom{8}{3}+\binom{7}{2}+\binom{6}{1}$ it follows that $\alpha_4\leq \binom{8}{4}+\binom{7}{3}+\binom{6}{2}=120$
						and $\alpha_5\leq \binom{8}{5}+\binom{7}{4}+\binom{6}{3}=111$. From \eqref{cucucu} it follows that
						$$\alpha_4\leq \frac{1}{3}(\alpha_5+252)\text{ and thus }\alpha_4\leq 121.$$						
						Also, from \eqref{cucucu2}, $\beta_6^7 \leq 7$ it is implied by 
						$$\alpha_6-2\alpha_5 \leq 206-3\alpha_4.$$
						Now, as $\alpha_5\leq 111$ it follows that $\alpha_4\leq 121$ and thus it is enough to show that $\alpha_6-2\alpha_5 \leq -157$.	
						If $\alpha_5\geq \binom{8}{5}+\binom{7}{4}+\binom{5}{3}+\binom{4}{2}+\binom{1}{1}=108$ then
						$$\alpha_6-2\alpha_5\leq f_6(8)+f_5(7)+f_4(5)+f_3(4)+f_2(1)=-158,$$
						as required. So we can assume $\alpha_5\leq 107$. This implies $\alpha_4\leq 119$	and thus it is enough to show that
						$\alpha_6-2\alpha_5 \leq -151$. 
						For $\alpha_5\geq \binom{8}{5}+\binom{7}{4}+\binom{5}{3}+\binom{2}{2}+\binom{1}{1}=103$, the above condition holds,
						so we may assume $\alpha_5\leq 102$. It follows that $\alpha_4\leq 118$ and it is enough to show that
						$\alpha_6-2\alpha_5 \leq -148$. For $\alpha_5=102$, this condition holds, so we may assume $\alpha_5\leq 101$ which
						implies $\alpha_4=117$. Also, it is enough to show that $\alpha_6-2\alpha_5 \leq -145$, which is satisfied for $\alpha_5\geq 99$.
						Now, assume $\alpha_5\leq 98$. It follows that $\alpha_4\leq 116$, a contradiction.	
											
			\item $\alpha_3=82$. Since $\beta_4^7\geq 0$ it follows that $\alpha_4\geq 113$. On the other hand, since 
			      $\alpha_3=\binom{8}{3}+\binom{7}{2}+\binom{5}{1}$ it follows that $\alpha_4\leq \binom{8}{4}+\binom{7}{3}+\binom{5}{2}=115$
						and $\alpha_5\leq \binom{8}{5}+\binom{7}{4}+\binom{5}{3}=101$.  From \eqref{cucucu} it follows that
						$$\alpha_4\leq \frac{1}{3}(\alpha_5+246)\text{ and thus }\alpha_4\leq 115.$$
						Also, from \eqref{cucucu2}, $\beta_6^7 \leq 7$ it is implied by 
						$$\alpha_6-2\alpha_5 \leq 202-3\alpha_4.$$
						Now, as $\alpha_5\leq 101$ it follows that $\alpha_4\leq 115$ and thus it is enough to show that 
						$\alpha_6-2\alpha_5 \leq -143$. For $\alpha_5\geq \binom{8}{5}+\binom{7}{4}+\binom{4}{3}+\binom{2}{2}+\binom{1}{1}=97$
						we have $$\alpha_6-2\alpha_5\leq f_6(8)+f_5(7)+f_4(4)+f_3(2)+f_2(1)=-144,$$
						and we are done. Assume $\alpha_5\leq 96$. Then $\alpha_4\leq 114$ and it is enough to to show that
						$\alpha_6-2\alpha_5 \leq -140$. This holds for $\alpha_5\geq 95$ so we may assume $\alpha_5\leq 94$. This implies
						$\alpha_4=113$ and it is enough to show that $\alpha_6-2\alpha_5 \leq -137$. This is true 
						for $\alpha_5\geq  \binom{8}{5}+\binom{7}{4}+\binom{3}{3}+\binom{2}{2}=93$. If $\alpha_5\leq 92$ then $\alpha_4\leq 112$,
						a contradiction.

			\item $\alpha_3=81$. Since $\beta_4^7\geq 0$ it follows that $\alpha_4\geq 109$. On the other hand, since 
			      $\alpha_3=\binom{8}{3}+\binom{7}{2}+\binom{4}{1}$ it follows that $\alpha_4\leq \binom{8}{4}+\binom{7}{3}+\binom{4}{2}=111$
						and $\alpha_5\leq \binom{8}{5}+\binom{7}{4}+\binom{4}{3}=95$.  From \eqref{cucucu} it follows that
						$$\alpha_4\leq \frac{1}{3}(\alpha_5+240)\text{ and thus }\alpha_4\leq 111.$$
						Also, from \eqref{cucucu2}, $\beta_6^7 \leq 7$ it is implied by 
						$$\alpha_6-2\alpha_5 \leq 198-3\alpha_4.$$
						Now, as $\alpha_5\leq 95$ and $\alpha_4\leq 111$, it is enough to show that $\alpha_6-2\alpha_5\leq -135$.
						This is true for $\alpha_5\geq \binom{8}{5}+\binom{7}{4}+\binom{3}{3}=92$.
						If $\alpha_5\leq 91$ then $\alpha_4\leq 110$ and it suffices to show that $\alpha_6-2\alpha_5\leq -132$.
						This is true for $\alpha_5\geq \binom{8}{5}+\binom{6}{4}+\binom{5}{3}+\binom{4}{2}+\binom{1}{1}=88$.
						If $\alpha_5\leq 87$ then $\alpha_4=109$ and it suffices to show that $\alpha_6-2\alpha_5\leq -129$.
						This is true for $\alpha_5\geq \binom{8}{5}+\binom{6}{4}+\binom{5}{3}+\binom{3}{2}+\binom{1}{1}=85$.
						If $\alpha_5\leq 84$ then $\alpha_4\leq 108$, a contradiction.
			
			\item $\alpha_3=80$. Since $\beta_4^7\geq 0$ it follows that $\alpha_4\geq 105$. On the other hand, since 
			      $\alpha_3=\binom{8}{3}+\binom{7}{2}+\binom{3}{1}$ it follows that $\alpha_4\leq \binom{8}{4}+\binom{7}{3}+\binom{3}{2}=108$
						and $\alpha_5\leq \binom{8}{5}+\binom{7}{4}+\binom{3}{3}=92$.  From \eqref{cucucu} it follows that
						$$\alpha_4\leq \frac{1}{3}(\alpha_5+234)\text{ and thus }\alpha_4 \leq 108.$$
						Also, from \eqref{cucucu2}, $\beta_6^7 \leq 7$ it is implied by 
						$$\alpha_6-2\alpha_5 \leq 194-3\alpha_4.$$
						Since $\alpha_5\leq 92$ and $\alpha_4\leq 108$ it is enough to show that $\alpha_6-2\alpha_5 \leq -130$.
						This holds for $\alpha_5\geq 85$, so we may assume $\alpha_5\leq 84$. It follows that $\alpha_4\leq 106$ 
						and it is enough to show that $\alpha_6-2\alpha_5 \leq -124$. This is true for 
						$\alpha_5\geq \binom{8}{5}+\binom{6}{4}+\binom{5}{3}+\binom{2}{2}=83$. If $\alpha_5\leq 82$ then $\alpha_4=105$
						and it is enough to show that $\alpha_6-2\alpha_5 \leq -121$. This holds for
						$\alpha_5\geq \binom{8}{5}+\binom{6}{4}+\binom{4}{3}+\binom{3}{2}+\binom{1}{1}=80$.
						Now, if $\alpha_5\leq 79$ then $\alpha_4\leq 104$, a contradiction.
			\end{itemize}

\item[(ii)] If $\alpha_2=35$ then $75\leq \alpha_3\leq 77$.
      \begin{itemize}
			\item $\alpha_3=77$. Since $\beta_4^7\geq 0$ it follows that $\alpha_4\geq 103$.
			      On the other hand, since $\alpha_3=\binom{8}{3}+\binom{7}{2}$ it follows that 
						$\alpha_4\leq \binom{8}{4}+\binom{7}{3}=105$
						and $\alpha_5\leq \binom{8}{5}+\binom{7}{4}=91$. From \eqref{cucucu} it follows that
						$$\alpha_4\leq \frac{1}{3}(\alpha_5+226)\text{ and thus }\alpha_4 \leq 105.$$
						Also, from \eqref{cucucu2}, $\beta_6^7 \leq 7$ it is implied by 
						$$\alpha_6-2\alpha_5 \leq 187-3\alpha_4.$$
						Since $\alpha_5\leq 91$ and $\alpha_4\leq 105$ it is enough to show that $\alpha_6-2\alpha_5 \leq -128$.
						If $\alpha_5\geq \binom{8}{5}+\binom{6}{4}+\binom{5}{3}+\binom{3}{2}=84$ then this condition holds.
						Assume $\alpha_5\leq 83$. We have $\alpha_4 =103 $ and it suffices to show that $\alpha_6-2\alpha_5 \leq -122$.
						This is true for $\alpha_5\geq \binom{8}{5}+\binom{6}{4}+\binom{4}{3}+\binom{3}{2}+\binom{1}{1}=79$.
						Now, if $\alpha_5\leq 79$ then $\alpha_4\leq 101$, a contradiction.
					
			\item $\alpha_3=76$. Since $\beta_4^7\geq 0$ it follows that $\alpha_4\geq 99$.
			      On the other hand, since $\alpha_3=\binom{8}{3}+\binom{6}{2}+\binom{5}{1}$ it follows that 
						$\alpha_4\leq \binom{8}{4}+\binom{6}{3}+\binom{5}{2}=100$
						and $\alpha_5\leq \binom{8}{5}+\binom{6}{4}+\binom{5}{3}=81$. From \eqref{cucucu} it follows that
						$$\alpha_4\leq \frac{1}{3}(\alpha_5+220)\text{ and thus }\alpha_4 \leq 100.$$
						Also, from \eqref{cucucu2}, $\beta_6^7 \leq 7$ it is implied by 
						$$\alpha_6-2\alpha_5 \leq 183-3\alpha_4.$$
						Since $\alpha_5\leq 81$ and $\alpha_4\leq 100$ it is enough to show that $\alpha_6-2\alpha_5 \leq -117$.
						This is true for $\alpha_5\geq \binom{8}{5}+\binom{6}{4}+\binom{4}{3}+\binom{2}{2}=76$.
						If $\alpha_5=75$ then $\alpha_4\leq 98$, a contradiction.
			
			\item $\alpha_3=75$. Since $\beta_4^7\geq 0$ it follows that $\alpha_4\geq 95$.
			      On the other hand, since $\alpha_3=\binom{8}{3}+\binom{6}{2}+\binom{4}{1}$ it follows that 
						$\alpha_4\leq \binom{8}{4}+\binom{6}{3}+\binom{4}{2}=96$
						and $\alpha_5\leq \binom{8}{5}+\binom{6}{4}+\binom{4}{3}=75$. From \eqref{cucucu} it follows that
						$$\alpha_4\leq \frac{1}{3}(\alpha_5+214)\text{ and thus }\alpha_4 \leq 96.$$
						Also, from \eqref{cucucu2}, $\beta_6^7 \leq 7$ it is implied by 
						$$\alpha_6-2\alpha_5 \leq 179-3\alpha_4.$$
						Since $\alpha_5\leq 75$ and $\alpha_4\leq 96$ it is enough to show that $\alpha_6-2\alpha_5 \leq -109$.
						This is true for $\alpha_5\geq \binom{8}{5}+\binom{6}{4}+\binom{3}{3}=72$. If $\alpha_5\leq 71$ then $\alpha_4=95$
						and it is enough to show that $\alpha_6-2\alpha_5 \leq -106$. This holds for
						$\alpha_5\geq \binom{8}{5}+\binom{5}{4}+\binom{4}{3}+\binom{3}{2}+\binom{1}{1}=69$.
						If $\alpha_5\leq 68$ then $\alpha_4\leq 94$, a contradiction.				
			\end{itemize}

\item[(iii)] If $\alpha_2=34$ then $70\leq \alpha_3\leq 71$.
      \begin{itemize}
			\item $\alpha_3=71$. Since $\beta_4^7\geq 0$ it follows that $\alpha_4\geq 89$.
			      On the other hand, since $\alpha_3=\binom{8}{3}+\binom{6}{2}$ it follows that 
						$\alpha_4\leq \binom{8}{4}+\binom{6}{3}=90$
						and $\alpha_5\leq \binom{8}{5}+\binom{6}{4}=71$.  From \eqref{cucucu} it follows that
						$$\alpha_4\leq \frac{1}{3}(\alpha_5+200)\text{ and thus }\alpha_4 \leq 90.$$
						Also, from \eqref{cucucu2}, $\beta_6^7 \leq 7$ it is implied by 
						$$\alpha_6-2\alpha_5 \leq 168-3\alpha_4.$$
						Since $\alpha_5\leq 71$ and $\alpha_4\leq 90$ it is enough to show that $\alpha_6-2\alpha_5 \leq -102$.
						This is true for $\alpha_5\geq \binom{8}{5}+\binom{5}{4}+\binom{4}{3}+\binom{2}{2}=66$. If 
						$\alpha_5\leq 65$ then $\alpha_4\leq 88$, a contradiction.
						
			\item $\alpha_3=70$. Since $\beta_4^7\geq 0$ it follows that $\alpha_4\geq 85$.
			      On the other hand, since $\alpha_3=\binom{8}{3}+\binom{5}{2}+\binom{4}{1}$ it follows that 
						$\alpha_4\leq \binom{8}{4}+\binom{5}{3}+\binom{4}{2}=86$
						and $\alpha_5\leq \binom{8}{5}+\binom{5}{4}+\binom{4}{3}=65$. From \eqref{cucucu} it follows that
						$$\alpha_4\leq \frac{1}{3}(\alpha_5+194)\text{ and thus }\alpha_4 \leq 86.$$
						Also, from \eqref{cucucu2}, $\beta_6^7 \leq 7$ it is implied by 
						$$\alpha_6-2\alpha_5 \leq 164-3\alpha_4.$$
						Since $\alpha_5\leq 65$ and $\alpha_4\leq 86$ it is enough to show that $\alpha_6-2\alpha_5 \leq -94$.
						This is true for $\alpha_5\geq \binom{8}{5}+\binom{5}{4}+\binom{3}{3}=62$. If $\alpha_5\leq 61$ then $\alpha_4=85$
						and it suffices to show that $\alpha_6-2\alpha_5 \leq -91$. This holds for $\alpha_5\geq 60$.
						If $\alpha_5\leq 59$ then $\alpha_4\leq 84$, a contradiction.
			\end{itemize}
			
\item[(iii)] If $\alpha_2=33$ then $65\leq \alpha_3\leq 66$.
      \begin{itemize}
			\item $\alpha_3=66$. Since $\beta_4^7\geq 0$ it follows that $\alpha_4\geq 79$.
			      On the other hand, since $\alpha_3=\binom{8}{3}+\binom{5}{2}$ it follows that 
						$\alpha_4\leq \binom{8}{4}+\binom{5}{3}=80$
						and $\alpha_5\leq \binom{8}{5}+\binom{5}{4}=61$. From \eqref{cucucu} it follows that
						$$\alpha_4\leq \frac{1}{3}(\alpha_5+180)\text{ and thus }\alpha_4 \leq 80.$$
						Also, from \eqref{cucucu2}, $\beta_6^7 \leq 7$ it is implied by 
						$$\alpha_6-2\alpha_5 \leq 153-3\alpha_4.$$
						Since $\alpha_5\leq 61$ and $\alpha_4\leq 80$ it is enough to show that $\alpha_6-2\alpha_5 \leq -87$.
						This is true for $\alpha_5\geq 58$. Assume $\alpha_5\leq 57$. Then $\alpha_4=79$ and it suffices
						to show $\alpha_6-2\alpha_5 \leq -84$, which is true for $\alpha_5\geq 56$. If $\alpha_5\leq 55$ then
						$\alpha_4\leq 78$, a contradiction.
						
			\item $\alpha_3=65$. Since $\beta_4^7\geq 0$ it follows that $\alpha_4\geq 75$.
			      On the other hand, since $\alpha_3=\binom{8}{3}+\binom{4}{2}+\binom{3}{1}$ it follows that 
						$\alpha_4\leq \binom{8}{4}+\binom{4}{3}+\binom{3}{2}=77$
						and $\alpha_5\leq \binom{8}{5}+\binom{4}{4}+\binom{3}{3}=58$. From \eqref{cucucu} it follows that
						$$\alpha_4\leq \frac{1}{3}(\alpha_5+174)\text{ and thus }\alpha_4 \leq 77.$$
						Also, from \eqref{cucucu2}, $\beta_6^7 \leq 7$ it is implied by 
						$$\alpha_6-2\alpha_5 \leq 149-3\alpha_4.$$
						Since $\alpha_5\leq 58$ and $\alpha_4\leq 77$ it is enough to show that $\alpha_6-2\alpha_5 \leq -82$.
						This is true for $\alpha_5\geq \binom{7}{5}+\binom{6}{4}+\binom{5}{3}+\binom{4}{2}=52$. If $\alpha_5\leq 51$
						then $\alpha_4=75$ and it is enough to show that $\alpha_6-2\alpha_5 \leq -76$. This is true for
						$\alpha_5\geq \binom{7}{5}+\binom{6}{4}+\binom{5}{3}+\binom{2}{2}=47$. On the other hand, $\alpha_5\geq 51$, so
						we are done.
			\end{itemize}
\end{enumerate}						

\item[(d)] $k=7$. For $2\leq k\leq 7$ we let $f_k(x)=\binom{x}{k}-\binom{x}{k-1}$. We have the table:
 \begin{center}
 \begin{table}[tbh]
 \begin{tabular}{|l|l|l|l|l|l|l|l|l|l|}
 \hline
 $x$      & 1  & 2  & 3  & 4  & 5  & 6  & 7   & 8    & 9    \\ \hline
 $f_7(x)$ & 0  & 0  & 0  & 0  & 0  & -1 & -6  & -20  & -48 \\ \hline
 $f_6(x)$ & 0  & 0  & 0  & 0  & -1 & -5 & -14 & -28  &   \\ \hline
 $f_5(x)$ & 0  & 0  & 0  & -1 & -4 & -9 &  -14 & -14  &  \\ \hline
 $f_4(x)$ & 0  & 0  & -1 & -3 & -5 & -5 & 0   &  14  & \\ \hline
 $f_3(x)$ & 0  & -1 & -2 & -2 & 0 & 5 &  14 & 28 &    \\ \hline
 $f_2(x)$ & -1 & -1 & 0  & 2  & 5 & 9 & 14 &  20 &    \\ \hline
 \end{tabular}
 \end{table}
 \end{center}

We recall that, in order to prove that $\beta_7^7\leq 8$, we can assume $\alpha_7\geq 12=\binom{8}{7}+\binom{6}{6}+\binom{5}{5}+\binom{4}{4}+\binom{3}{3}$.
This implies:
\begin{equation}\label{lime}
\alpha_2\geq 35,\;\alpha_3\geq 77,\;\alpha_4\geq 105,\;\alpha_5\geq 90,\;\alpha_6\geq 46.
\end{equation}
Using \eqref{lime}, we deduce that if $\alpha_2=35=\binom{8}{2}+\binom{7}{1}$ then
\begin{equation}\label{lime2}
\alpha_3=77,\; \alpha_4=105,\; \alpha_5\leq 91,\; \alpha_6\leq 49\text{ and }\alpha_7\leq 15.
\end{equation}
From \eqref{lime2} it follows that
$$\beta_7^7 \leq 15-46+91-105+77-35+9-1 = 5 \leq 8,$$
as required. Hence, we can assume $\alpha_2=36$ and $80\leq \alpha_3\geq 84$.
We consider several cases:
\begin{enumerate}
\item[(a)] $\alpha_3=80=\binom{8}{3}+\binom{7}{2}+\binom{3}{1}$. It follows that:
           $$\alpha_4\leq 108,\;\alpha_5\leq 92,\;\alpha_6\leq 49\text{ and }\alpha_7\leq 15.$$
					 Since, from \eqref{lime}, we have $\alpha_4\geq 105$ and $\alpha_6\geq 46$, it follows that
					 $$\beta_7^7 \leq 15-46+92-105+80-36+9-1= 8,$$
					 and we are done.
\item[(b)] $\alpha_3=81=\binom{8}{3}+\binom{7}{2}+\binom{4}{1}$. It follows that:
           $$\alpha_4\leq 111,\;\alpha_5\leq 95,\;\alpha_6\leq 50,\;\alpha_7\leq 15.$$					
					 On the other hand, since $\beta_4^7\geq 0$ it follows that $\alpha_4\geq 109$.					
					 Therefore
					 $$\beta_7^7 \leq 15-46+95-109+81-36+9-1= 8,$$
					 and we are done.
\item[(c)] $\alpha_3=82=\binom{8}{3}+\binom{7}{2}+\binom{5}{1}$. It follows that:
           $$\alpha_4\leq 115,\;\alpha_5\leq 101,\;\alpha_6\leq 54,\;\alpha_7\leq 16.$$
					 Since $\beta_4^7\geq 0$ it follows that $\alpha_4\geq 113$.
					 
					 If $\alpha_7=12$ then 				
					 $$\beta_7^7 \leq 12 - 46 + 101 - 113 + 82 - 36 + 9 - 1 = 8.$$
					 If $\alpha_7\in \{13,14\}$ then $\alpha_6\geq 48$ and, as above, we get $\beta_7^7 \leq 8$.
					
					 If $\alpha_7 = 15$ then $\alpha_6\geq 49$ and we obtain again $\beta_7^7\leq 8$. 
					
					 Finally, if $\alpha_7=16$ then $\alpha_6\geq 54$ and we are also done.	

\item[(d)] $\alpha_3=83=\binom{8}{3}+\binom{7}{2}+\binom{6}{1}$. It follows that:
           $$\alpha_4\leq 120,\;\alpha_5\leq 111,\;\alpha_6\leq 64,\;\alpha_7\leq 21.$$
					 Since $\beta_4^7\geq 0$ it follows that $\alpha_4\geq 117$. Note that 
					 $$\beta_5^5 \leq 111-117+83-36+9-1 = 49,$$
					 so, in order to prove that $\beta_7^7=(\alpha_7-\alpha_6)+\beta_5^5\leq 8$, it suffices
					 to show that $\alpha_7-\alpha_6\leq -41$. If $\alpha_6\geq \binom{8}{6}+\binom{7}{5}+\binom{5}{4}+\binom{3}{3}+\binom{2}{2}+\binom{1}{1}=57$
					 then $$\alpha_7-\alpha_6\leq f_7(8)+f_6(7)+f_5(5)+f_4(3)+f_3(2)+f_2(1)=-41,$$
					 and we are done. So, we may assume $\alpha_6\leq 57$. Since $\beta_6^7\geq 0$, it follows that
					 $$2\alpha_5 \leq 57 + 3\cdot 117 -4 \cdot 83 + 5\cdot 36 -6\cdot 9 + 7 = 209,$$
					 that is $\alpha_5\leq 104$. Therefore, $\beta_5^5\leq 42$, so, in order to prove that $\beta_7^7\leq 8$, it suffices to show
					 that $\alpha_7-\alpha_6\leq -34$. This holds for 
					 $\alpha_6\geq \binom{8}{6}+\binom{6}{5}+\binom{5}{4}+\binom{4}{3}+\binom{2}{2}+\binom{1}{1}=45,$
					 and we are done.

\item[(e)] $\alpha_4=84$. Since $\beta_4^7\geq 0$ it follows that $\alpha_4\geq 121$. 
            We have the following table:
											
						\begin{center}
						\begin{table}[tbh]
            \begin{tabular}{|l|l|l|l|l|l|l|}
						\hline
						$\alpha_4$                  & 121 & 122 & 123 & 124 & 125 & 126 \\ \hline
						$\max\{\alpha_5\}$          & 111 & 112 & 114 & 117 & 121 & 126 \\ \hline
						$\max\{\alpha_5-\alpha_4\}$ & -10 & -10 & -9  & -7  & -4  & 0   \\ \hline
						\end{tabular}
						\end{table}
            \end{center}
						
						Note that, in order to prove that $\beta_7^7\leq 8$, it is enough to show that 
						$$\alpha_7-\alpha_6\leq -48-\max\{\alpha_5-\alpha_4\}.$$
            We consider several subcases:
						\begin{itemize}
						\item $\alpha_4=121$. Since $\beta_5^7\geq 0$ it follows that $\alpha_5\geq 105$.						
						      We have to show that $\alpha_7-\alpha_6\leq -38$. This is true for
									$\alpha_6\geq \binom{8}{6}+\binom{7}{5}+4=53$, so we may assume that $\alpha_6\leq 52$.
									Now, since $\alpha_6\leq 52$ and $\beta_6^7\geq 0$ it follows that $\alpha_5\leq 106$.
									Since $\alpha_5-\alpha_4\leq -15$, we have to show that $\alpha_7-\alpha_6\leq -33$.
									This is true for $\alpha_6\geq \binom{8}{6}+\binom{6}{5}+\binom{5}{4}+\binom{4}{3}+\binom{2}{2}=44$,
									so we are done.
       			\item $\alpha_4=122$. Since $\beta_5^7\geq 0$ it follows that $\alpha_5\geq 108$.
						      As in the previous case, we have to show that $\alpha_7-\alpha_6\leq -38$, so we can assume 
									that $\alpha_6\leq 52$. It follows that 
									$$\beta_6^7 \leq 52 - 2\cdot 108 + 3\cdot 122 - 4\cdot 84 + 5\cdot 36 - 6\cdot 9 + 7 =-1,$$
									a contradiction.
									
						\item $\alpha_4=123$. Since $\beta_5^7\geq 0$ it follows that $\alpha_5\geq 111$.
						      We have to show $\alpha_7-\alpha_6\leq -39$. This is true for
									$\alpha_6\geq \binom{8}{6}+\binom{7}{5}+\binom{5}{4}+\binom{3}{3}=55$.
									Now, assume $\alpha_6\leq 54$. It follows that 
									$$\beta_6^7 \leq 54 - 2\cdot 111 + 3\cdot 123 - 4\cdot 84 + 5\cdot 36 - 6\cdot 9 + 7 =-2,$$
									a contradiction.
									
						\item $\alpha_4=124$. Since $\beta_5^7\geq 0$ it follows that $\alpha_5\geq 114$.
						      We have to show $\alpha_7-\alpha_6\leq -41$. This is true for
									$\alpha_6\geq \binom{8}{6}+\binom{7}{5}+\binom{5}{4}+\binom{3}{3}+\binom{2}{2}+\binom{1}{1}=57$.
									Thus we can assume $\alpha_6\leq 56$. It follows that 
									$$\beta_6^7 \leq 56 - 2\cdot 114 + 3\cdot 124 - 4\cdot 84 + 5\cdot 36 - 6\cdot 9 + 7 =-3,$$
									a contradiction.
						
            \item $\alpha_4=125$. Since $\beta_5^7\geq 0$ it follows that $\alpha_5\geq 117$.
						      We have to show $\alpha_7-\alpha_6\leq -44$. This is true for 
									$\alpha_6\geq \binom{8}{6}+\binom{7}{5}+\binom{5}{4}+\binom{4}{3}+\binom{3}{2}+\binom{2}{1}=63$ 									
                  Thus, we can assume $\alpha_6\leq 62$. It follows that 
									$$\beta_6^7 \leq 62 - 2\cdot 117 + 3\cdot 125 - 4\cdot 84 + 5\cdot 36 - 6\cdot 9 + 7 =0.$$
									Therefore, if $\alpha_5>117$ or $\alpha_6<62$ then $\beta_6^7<0$, a contradiction. Thus $\alpha_5=117$ and 
									$\alpha_6=62=\binom{8}{6}+\binom{7}{5}+\binom{5}{4}+\binom{4}{3}+\binom{3}{2}+\binom{1}{1}$,
									and thus $\alpha_7\leq 8+7+3=18$. Hence $\beta_7^7\leq 4$ and we are done.								
						
            \item $\alpha_4=126$.	Since $\beta_5^7\geq 0$ it follows that $\alpha_5\geq 120$.														
						      We have to show $\alpha_7-\alpha_6\leq -48$. This is true for 
									$\alpha_6\geq \binom{8}{6}+\binom{7}{5}+\binom{6}{4}+\binom{4}{3}+\binom{2}{2}+\binom{1}{1}=66$ 									
                  Thus, we can assume $\alpha_6\leq 65$. It follows that 
									$$\beta_6^7 \leq 65 - 2\cdot 120 + 3\cdot 126 - 4\cdot 84 + 5\cdot 36 - 6\cdot 9 + 7 =0.$$
									Therefore $\alpha_5=120$ and $\alpha_6=65=\binom{8}{6}+\binom{7}{5}+\binom{6}{4}+\binom{4}{3}+\binom{2}{2}$.
									It follows that $\alpha_7\leq 8+7+6+1=22$. Hence $\beta_7^7\leq 7$ and we are done.
						\end{itemize}
\end{enumerate}
\end{enumerate}
Thus, the proof is complete.
\end{proof}

\begin{teor}\label{main3}
 Let $I\subset S=K[x_1,\ldots,x_n]$ be a squarefree monomial ideal, where $n\leq 9$. Then 
 $\hdepth(I)\geq \hdepth(S/I)$.
\end{teor}

\begin{proof}
If $I$ is principal, which according to Theorem \ref{teo1}, is equivalent to $q:=\qdepth(S/I)=n-1$, we
have, again, by Theorem \ref{teo1} that $\qdepth(I)=n$ and there is nothing to prove. Hence, we may assume that $q=\qdepth(S/I)\leq 6$.
If $q\leq 4$ we are done by \cite[Theorem 3.10]{bordi} and \cite[Theorem 3.15]{bordi}. For $q=5$, the conclusion follows from \cite[Theorem 4.4]{bordi}. Also, for $q=6$, the conclusion follows from Theorem \ref{main2}. Finally, if $q=7$ then $n=9$ and the conclusion follows from Lemma \ref{lem} and Lemma \ref{79}.
\end{proof}

\subsection*{Data availability}

Data sharing not applicable to this article as no data sets were generated or analyzed
during the current study.

\subsection*{Conflict of interest}

The authors have no relevant financial or non-financial interests to disclose.


\end{document}